\documentclass[12pt]{amsart}

\usepackage[latin1]{inputenc}
\usepackage[english]{babel}
\usepackage{graphicx}
\usepackage{amsmath,amsfonts,amssymb,amsthm}
\usepackage[pdfauthor={Kleyber Cunha and Daniel Smania},%
pdftitle={Universality and Rigidity for piecewise smooth homeomorphisms on the circle},%
pdfsubject={In this work we study the renormalization operator acting on piecewise smooth homeomorphisms on the circle, that turns out to be essentially the study of Rauzy-Veech renormalizations of generalized interval exchanges maps  with genus zero. In particular we show that renormalizations of such maps with  zero mean nonlinearity and satisfying certain smoothness and combinatorial assumptions converges to the set of  piecewise affine interval exchange maps.   MSC 37E10, 37E05, 37E20, 37C05, 37B10 },%
pdfkeywords={renormalization, interval exchange transformations, Rauzy-Veech induction, universality, homeomorphism on the circle, convergence}]{hyperref}
\usepackage{latexsym}
\usepackage{epsf}
\usepackage{epstopdf}
\usepackage{psfrag}
\usepackage{indentfirst}
\usepackage{eucal}
\usepackage{fontenc}
\usepackage{float}
\usepackage[all]{xy}

\def\ve{\varepsilon}
\def\al{\alpha}
\def\A{\mathcal{A}}
\def\T{\mathcal{T}}
\def\la{\lambda}

\def\nor{\mathrm{nor}}
\def\im{\mathrm{Im}}
\def\d{\mathrm{d}}
\def\O{\mathrm{O}}
\def\Om{\Omega}

\def\Z{\mathcal{Z}}
\def\R{\mathbb{R}}
\def\N{\mathbb{N}}
\def\rot{\mathrm{rot}}
\def\T{\mathcal{T}}
\def\Hom{\mathrm{Hom}}
\newcommand{\dem}{\begin{proof}}
\newcommand{\cqd}{\end{proof}}

\theoremstyle{plain}
\newtheorem{lemma}{\bf Lemma}[section]

\newtheorem{rem}{Remark}[section]
\newtheorem{pro}[lemma]{Proposition}
\newtheorem{maintheorem}{Theorem}

\theoremstyle{definition}

\title{Universality and Rigidity for piecewise smooth homeomorphisms on the circle}
\author{Kleyber Cunha and Daniel Smania}

\address{
Departamento de Matem\'atica,
Universidade Federal da Bahia, Av. Ademar de Barros s/n,
CEP 40170-110
Salvador-BA, Brazil}
\email{kleyber@ufba.br}

\address{
Departamento de Matem\'atica,
ICMC-USP, Caixa Postal 668,  S\~ao Carlos-SP,
CEP 13560-970
S\~ao Carlos-SP, Brazil}
\email{smania@icmc.usp.br}
\urladdr{www.icmc.usp.br/$\sim$smania/}

\date{\today} 
\thanks{ This work is based on the Ph. D. Thesis of the first author.    K. C was partially supported by FAPESP 07/01045-7.
D.S. was partially supported by FAPESP 2008/02841-4 and 2010/08654-1, CNPq 310964/2006-7 and 303669/2009-8.}

\begin{document}

\maketitle

\begin{abstract}
\footnotesize{In this work, we find sufficient conditions for two piecewise $C^{2+\nu}$ homeomorphism $f$ and $g$ of the circle to be $C^1$ conjugate. Besides the restrictions on the combinatorics of the maps (we assume that maps have bounded ``rotation number'' ), and necessary conditions on the one-side derivatives of points where $f$ and $g$ are not differentiable, we also assume zero mean nonlinearity for $f$ and $g.$

The proof is based on the study of Rauzy-Veech renormalization of genus one generalized interval exchange maps with certain restrictions on its combinatorics. }
\end{abstract}
\bigskip

\medskip
\tableofcontents


\section{Introduction and results}
\label{}
Consider the following rigidity problem.  Let $f,g :X\to X$ be two (piecewise) smooth dynamical systems that are conjugated by an orientation  preserving homeomorphism, i.e., there is $h:X\to X$ such that $f\circ h =h\circ g.$ On what conditions is  the conjugation smooth (for instance $C^1$)? When $X=\mathbb{S}^1$ and $f,g$ are smooth diffeomorphisms there are many results of rigidity, for example, \cite{herman}, \cite{sinai}, \cite{yoc}, \cite{khanin1}, \cite{katz}. In this article we study the rigidity problem for piecewise smooth homemorphisms on the circle.

The map $f:\mathbb{S}^1\to\mathbb{S}^1$ is a piecewise smooth homemomorphism on the circle if $f$ is a homeomorphism, has jumps in the first derivative on finitely many points, that we call break points, and $f$ is smooth outside its break points. The set $BP_f=\{x\in\mathbb{S}^1:\; BP_f(x):=Df(x_-)/ Df(x_+)\neq 1\}$ is called the set of break points of $f$ and the number $BP_f(x)$ is called the break of $f$ at $x.$  Denote $BP_f=\{x_1,...,x_m\}$ and $BP_g=\{y_1,...,y_n\}.$

We say that two piecewise smooth homeomorphisms on the circle are {\it break-equivalents} if  there exists a topological conjugacy $h$ such that  $h(BP_f)=BP_g$ and   $BP_f(x_i)=BP_g(h(x_i))$. It is easy to see that  there is a $C^1$  conjugacy between $f$ and $g$ then  $f$ and $g$ are {\it break-equivalents}.

As in \cite{cunhasmania} the key  idea is to consider piecewise smooth homeomorphisms on the circle as  generalized interval exchange transformations, g.i.e.t. for short. Let $I$ be an interval  and let  $\A$ be a finite set (the {\it  alphabet}) with $d~\geq~2$ elements and $\mathcal{P} =\{I_\al: \al\in\A\}$ be an $\mathcal{A}-${\it indexed} partition of $I$ into subintervals\footnote{All the subintervals will be bounded, closed on the left and open on the right.}. We say that the triple $(f,\mathcal{A}, \mathcal{P})$, where  $f:I\to I$ is a bijection, is a g.i.e.t. with $d$ intervals, if $f|_{I_\al}$ is an orientation-preserving homeomorphism for each $\alpha \in \mathcal{A}$. The order of subinterval in the domain and in the image of $f$ constitue the combinatorial data of $f,$ denoted by $\pi=(\pi_0,\pi_1),$ where $\pi_i: \A\to \{1,...,d\}$ is a bijection for $i=1,2$ and $\pi_0, \pi_1$ given the order of subintervals $I_\al$ and $f(I_\al)$ in $I,$ respectively. For the explicit formula of $\pi=(\pi_0,\pi_
1)$ see \cite{cunhasmania}. We always assume that the combinatorial data is irreducible, i.e., $\pi_1\circ\pi_0^{-1}(\{1,...,s\})\neq\{1,...,s\}$ for all $1\leq s \leq d-1.$

There is a renormalization scheme in the space of g.i.e.m. called the Rauzy-Veech induction (\cite{rauzy}, \cite{veech}), that  associates with $f$ a sequence of first return maps $f_n=R^n(f)$ to  a nested sequence of  intervals $I^n$ with the same left endpoint of  $I.$ The map $f_n$ is again g.i.e.m. with the same alphabet $\A$ but the combinatorial data may be different.

More specifically, denoting by $\al(0),\al(1)\in\A$ the letters such that $\pi_0(\al(0))=d$ and $\pi_1(\al(1))=d,$ we compare the lenght of the intervals $I_{\al(0)}$ and $f(I_{\al(1)}).$ If $|I_{\al(0)}|>|f(I_{\al(1)})|$ (resp. $|I_{\al(0)}|>|f(I_{\al(1)})|$ ) we say that $f$ has type 0 (resp. type 1) and that the letter $\al(0)$ is the winner (resp. loser) and that the letter $\al(1)$ is the loser (resp. winner). Then putting $I^1=I\setminus f(I_{\al(1)})$ (resp. $I^1=I\setminus I_{\al(0)}$), we have that $R(f)$ is the first return map of $f$ to the interval $I^1$ and so on.   We say that $f$ has no connections if the orbits of the boundary of each $I_\al$ are distinct whenever possible.  It  has been  established by \cite{keane} that if  an interval exchange transformation $f$ has no connections then $f$ is infinitely renormalizable. If a g.i.e.m. $f$ is infinitely renormalizable  then the  sequence $(\pi^n,\ve^n)_n$ of combinatorial data and types of $R^n(f)$ is called {\it combinatorics} of $f.$ For 
more details about the Rauzy-Veech induction the reader may consult, for example, \cite{cunhasmania}, \cite{yoccoz}, \cite{viana}.

For each i.e.t. $f$ it is  possible to associate a genus $g$ that corresponds to the genus of translate surface associate to $f$ \cite{zorich}. This genus is invariant under Rauzy-Veech renormalization. Indeed  $f$ has {\it genus one} if $f$ has at most two discontinuities.  In a similar way, we will say that a g.i.e.m.  has genus one if $f$ has at most two discontinuites.  If $f$ is a homeomorphism on the circle, then $f$ has genus one as a g.i.e.m.

Let $H$ be  a non-degenerate interval, let  $g:H\to\mathbb{R}$ be a homeomorphism  and let $J\subset H$ be an interval. We define the {\it Zoom} of $g$ in $H,$  denoted by $\Z_H(g),$ the transformation $\Z_H(g)=A_1\circ g\circ A_2,$ where $A_1$ and $A_2$ are orientation-preserving affine maps, which sends $[0,1]$ into $H$ and $g(H)$ into $[0,1]$ respectively. So we can identify  a  g.i.e.m.  $f$ defined in the interval $[0,1]$ with the  quadruple
$$\big(\pi, (|I_\al|)_{\al\in\A}, (|f(I_\al)|)_{\al\in\A},(Z_{I_\al}f)_{\al\in\A}\big)\in \Pi_d\times \Delta_{\mathcal{A}}\times\Delta_{\mathcal{A}}\times \Hom_+([0,1]),$$
where $\Hom_+([0,1])$ is the set of  orientation preserving homeomorphisms $h:[0,1]\to[0,1]$ such that $h(0)=0$ and $h(1)=1,$  $|J|$ denote the length of interval $J$ and

$$\Delta_{\mathcal{A}}=\{(x_\alpha)_{\alpha\in \mathcal{A}} \in \mathbb{R}_+^{\mathcal{A}} \text{ s.t. } x_\alpha> 0 \text{ and } \sum_\alpha x_\alpha=1  \}.$$

 If we change the set  $\Hom_+([0,1])$ by the set $\mathrm{Diff}^r_+([0,1])$ of orientation preserving diffeomorphism $h$ of class $C^r$ such that $h(0)=0$ and $h(1)=1$ we get the space of g.i.e.m. of class $C^r.$ If we change the set $\Hom_+([0,1])$ by the set $\{\mathrm{Id}\}$ we have the space of affine i.e.t.. If in addition we replace  the set $ \Delta_{\mathcal{A}}\times\Delta_{\mathcal{A}}$ by the set $\{(\la,\la), \la\in\Delta_{\mathcal{A}}\}$ we get the space of standard i.e.t..

In the set of g.i.e.t. of class $C^r$, $r\geq0$, we define the distance in the $C^r$ topology by
\begin{eqnarray}
\label{distancia}
d_{C^r}(f,g)&:=&\max_{\al\in\A}\left\{\|Z_{I_\al}f-Z_{\tilde{I}_\al}g\|_{C^r}\right\} +\nonumber\\
& & \left\|(I_\al)_{\al\in\A}-(\tilde{I_{\al}})_{\al\in\A}\right\|_1+  \left\|(f(I_\al))_{\al\in\A}-(g(\tilde{I_{\al}}))_{\al\in\A}\right\|_1,
\end{eqnarray}
where $\|\cdot\|_{C^r}$ denote the sup-norm in the $C^r$ topology and $\|\cdot\|_1$ denote the sum-norm.

Let  $\al^n(\ve^n), \al^n(1-\ve^n)$ be the   winner and loser letters  of $R^n(f)$, where $\ve^n\in\{0,1\}$ is its  type.
As defined in \cite{cunhasmania} we say that infinitely renormalizable g.i.e.m. $f$ has $k-${\it bounded combinatorics} if for each $n$ and $\beta, \gamma  \in \mathcal{A}$ there exists $n_1, p \geq 0$, with $|n-n_1|< k$ and $|n-n_1-p|< k$, such that
$\alpha^{n_1}(\ve^{n_1})=\beta$, $\alpha^{n_1+p}(1-\ve^{n_1+p})=\gamma$ and $$\alpha^{n_1+i}(1-\ve^{n_1+p})= \alpha^{n_1+i+1}(\ve^{n_1+i})$$
for every $0\leq i < p$.

Let $\mathcal{B}_k^{2+\nu}$, $k \in \mathbb{N}$ and  $\nu > 0$,  be the set of g.i.e.m.  $f:I\to I$ such that
\begin{itemize}
\item[(i)]  For each $\alpha \in \mathcal{A}$ we can extend $f$ to $\overline{I_\al}$ as an orientation-preserving diffeomorphism of class $C^{2+\nu}$;\\
\item[(ii)] the g.i.e.m. $f$  has  $k-$bounded combinatorics;\\
\item[(iii)] The map   $f$ has  genus one and has no connections;
\end{itemize}

In \cite{cunhasmania} the authors show that if $f\in\mathcal{B}_k^{2+\nu}$ then $R^n(f)$ converges to a $3(d-1)-$dimensional space of the fractional linear g.i.e.t. Moreover if mean-nonlinearity is zero then $R^n(f)$ converge to a $2(d-1)-$dimensional space of the affine interval exchange maps. Our main results are:
 \begin{maintheorem}
\label{teo1}
Let $(f,\mathcal{A}, \{ I_\alpha\}_{\alpha \in \mathcal{A}}) \in \mathcal{B}_k^{2+\nu}$ be such that
 \begin{equation}\label{nonlin0}  \int \frac{D^2f(x)}{Df(x)} dx =0.\end{equation}
Then there exists an affine i.e.t. $(f_A,\mathcal{A}, \{ \tilde{I}_\alpha\}_{\alpha \in \mathcal{A}}),$ i.e., $f_A|_{\tilde{I}_\alpha}$ is affine for each $\alpha \in \mathcal{A},$ and $0<\la<1$ such that
\begin{itemize}
 \item[(i)] $f_A$ has the same combinatorics of $f;$
 \item[(ii)] $d_{C^2}(R^nf,R^nf_A)=\O(\lambda^{\sqrt{n}}).$
 \end{itemize}
\end{maintheorem}

\begin{maintheorem}[Universality] If $f$ and $g$ satisfies the assumptions of Theorem \ref{teo1},  they have  the same combinatorics and they are break-equivalents then we can choose $f_A=g_A.$
\label{teo2}
\end{maintheorem}

The next result is a consequence of Theorem \ref{teo1} and Theorem \ref{teo2}.

\begin{maintheorem}\label{teo66}
Let $f,g\in\mathcal{B}^{2+\nu}_k$ be  such that
\begin{itemize}
\item[i.] $f$ and $g$ have the same combinatorics;
\item[ii.] $f$ and $g$ are break-equivalents;
\item[iii.] We have $$\int_0^1\frac{D^2f(s)}{Df(s)}ds = \int_0^1\frac{D^2g(s)}{Dg(s)}ds=0.$$
\end{itemize}
Then there exists $0<\la<1$ such that
$$d_{C^2}(R^nf,R^ng)=\O(\lambda^{\sqrt{n}}).$$
\label{teo3}
\end{maintheorem}

It is known that if $f$ and $g$ has the same $k-$bounded combinatorics then they are semi-conjugate and this semi-conjugation sends break-point in break-point \cite{yoccoz}. If $f$ and $g$ have genus one then this semi-conjugation is indeed a conjugation (non wandering intervals).

\begin{maintheorem}[Rigidity] \label{teo4} Suppose that  $f$ and $g$ satisfy the assumptions of Theorem 3. Then $f$ and $g$ are $C^1$-conjugated.
\end{maintheorem}


\begin{maintheorem}[Linearization]\label{teo5}
If  $f$  satisfies the assumptions of Theorem \ref{teo1} then  $f$ is $C^1$-conjugate with a unique piecewise affine  homeomorphism on the circle.

\end{maintheorem}

There are previous results  on rigidity for piecewise smooth diffeomorphism of the circle with only one break point  and also  satisfying certain combinatorial restrictions \cite{kk} \cite{kt}.  There are also recent results \cite{yoccoz} on the structure of the set of "simple", small  deformations of a standard i.e.m.  $T_0$ (with certain Roth type combinatorics) which are $C^r$ conjugated with $T_0$,  but the nature  of their results and methods are quite distinct from ours.

\begin{rem} As notice by the anonymous referee, the estimates for the rate of convergence  of the renormalization operator that appears above are  not optimal. We would expect, as commented in \cite{sinai} in the case of diffeomorphisms on the circle, that  in fact exponential convergence holds true.\end{rem}

\section{Rauzy-Veech Cocycle}
In this section we use the notation  of  \cite{viana}.

Let $\pi$ be a combinatorial data of a g.i.e.m. and let  $\lambda=(\lambda_\alpha)_{\alpha \in \mathcal{A}}$ be a vector in $\mathbb{R}^{\mathcal{A}}$. Define $\omega=(\omega_{\al})_{\al\in\A}$  as
\begin{equation}\label{tau} \omega_{\al}=\sum_{\pi_1(\beta)<\pi_1(\alpha)} \lambda_{\beta}-\sum_{\pi_0(\beta)<\pi_0(\alpha)}\lambda_{\beta}.\end{equation}
 Notice that $\Omega_{\pi}(\la)=\omega$, where   $\Omega_{\pi}$ is the anti-symmetric matrix given by
$$
\Omega_{\al,\beta}=\left\{
\begin{array}{ll}
+1& \mbox{~se~} \pi_1(\al)>\pi_1(\beta) \mbox{~and~} \pi_0(\al)<\pi_0(\beta)\\
-1& \mbox{~se~} \pi_1(\al)<\pi_1(\beta) \mbox{~and~} \pi_0(\al)>\pi_0(\beta)\\
0& otherwise.
\end{array}
\right.
$$
If $\pi$ has genus one then $\dim Ker \ \Omega_\pi=  d-2$, so $\dim Im \ \Omega_\pi =2$. Denote by $\Pi^1$ the set of all possible genus one irreducible combinatorial data $\pi=(\pi_0,\pi_1)$. The Rauzy-Veech cocycle are the functions
$$\Theta_\ve\colon \Pi^1 \times \mathbb{R}^\mathcal{A} \rightarrow  \Pi^1 \times \mathbb{R}^\mathcal{A},$$
with $\ve \in \{0,1\}$, defined by  $\Theta_\ve(\pi,v)=(r_\ve(\pi),\Theta_{\pi,\epsilon}v)$. Here
\begin{eqnarray}
\label{B}
\Theta_{\pi,\ve}=\mathbb{I}+ E_{\al(1-\ve)\al(\ve)},
\end{eqnarray}
where $E_{\al\beta}$ is the elementary matrix whose only nonzero coefficient is $1$ in position $(\al,\beta)$ and $r_\ve(\pi)$ is the combinatorial data of $R(f).$

 We know that if $\pi'= r_\ve(\pi)$ then

\begin{equation}\label{semiconj}         \Theta_{\pi,\ve} \Omega_{\pi} = \Omega_{\pi'} (\Theta_{\pi,\ve}^t)^{-1}\end{equation}

Let $g:[0,1)\to[0,1)$ be an affine i.e.m. without connexions. Then $g$ is uniquely determined by the triple $(\pi,\la,\omega^0),$ where $\pi$ is the combinatorial data, $\la=(\la_\al)_{\al\in\A}\in\mathbb{R}^d_+$ is the partition vector of the domain and $\omega^0=(\omega_{\al}^0)_{\al\in\A}\in\mathbb{R}^d$ is such that
$$g(x)=e^{\omega_\al^0}x+\delta_{\al}\mbox{~ for all~}x\in I_{\al}.$$

For each $n$ denote by $\omega^n=(\omega^n_{\al})_{\al}$ the vector such that $R^n(g)(x)=e^{\omega^n_{\al}}x+\delta^n_{\al}$ for all $x\in I^n_{\al}.$ By Rauzy-Veech algorithm we know that

\begin{eqnarray}
\label{der1}\omega^{n+1}_{\al}&=&\omega^n_{\al} \mbox{~if~} \al\not=\al^n(1-\varepsilon)\\
\label{der2}\omega^{n+1}_{\al^{n}(1-\varepsilon)}&=&\omega^n_{\al^n(\varepsilon)}+\omega^n_{\al^n(1-\varepsilon)}, \text{ otherwise,}
\end{eqnarray}
where $\al^{n}(\varepsilon)$ and $\al^{n}(1-\varepsilon)$ are the winner and loser of $R^n(g)$ respectively. Therefore

$$\Theta_n(\omega^n)=\Theta_{\pi^n,\varepsilon^n}(\omega^n)=\omega^{n+1}.$$
Repeating this process inductively we have

$$\Theta_n\Theta_{n-1}\cdots \Theta_1\Theta_0(\omega^0)=\omega^n.$$

To prove  Theorem \ref{teo1} we need to understand the hyperbolic properties of the Rauzy-Veech cocycle restricted to the $k$-bounded combinatorics.

\subsection{Invariant cones}  Since $\Theta_{\pi,\ve}$ , $\Theta_{\pi,\ve}^t$  are non negative matrices, it preserves the positive cone $\mathbb{R}_+^\mathcal{A}$.  It  follows from  \eqref{semiconj} that
$$ \Theta_{\pi,\ve} \im \ \Omega_{\pi} = \im \ \Omega_{\pi'}.$$

We need to find  cones inside  $\im \ \Omega_{\pi}$ which are invariant by the action of  $\Theta_{\pi, \ve}$ and  $\Theta_{\pi, \ve}^{-1}$.
Define  the two dimensional cone $$C^s_\pi := \Omega_{\pi} \mathbb{R}^\mathcal{A}_+\subset \im \ \Omega_{\pi}.$$ It follows from \eqref{semiconj} that
 $$     \Theta_{\pi,\ve}^{-1} C^s_{\pi'} \subset C^s_{\pi}.$$
%
%
%
%
%

 For each $\pi \in \Pi^1$ define the convex cone
 $$T^+_\pi=\{ (\tau_\alpha)_{\alpha \in \mathcal{A}} \colon \ \sum_{\pi_0(\alpha)\leq k}  \tau_\alpha > 0 \ \ and  \  \sum_{\pi_1(\alpha) \leq k} \tau_\alpha  <  0,  \text{ for every } 1\leq k\leq d-1  \}$$

We have  \cite[Lemma 2.13]{viana} that $$(\Theta^t_{\pi,\ve})^{-1} T^+_\pi \subset  T^+_{\pi'}$$ Define
$$C^u_\pi= - \Omega_\pi T^+_\pi \subset  Im \ \Omega_\pi $$
By definition
$$C^u_\pi\subset  \ \im \ \Omega_\pi \cap  \mathbb{R}_+^\mathcal{A},$$
and it is easy to show that
\begin{equation}\label{empker} C^u_\pi\cap \ker \ \Omega_\pi =\{  0\}.\end{equation}
Note that   $ \Theta_{\pi,\ve}  C^u_\pi \subset C^u_{\pi'}$. Indeed applying $T^+_\pi$ in  \eqref{semiconj}  we have that

     $$ \Theta_{\pi,\ve} (-\Omega_{\pi} T^+_\pi)= -\Omega_{\pi'} (\Theta_{\pi,\ve}^t)^{-1}T^+_\pi\subset -\Omega_{\pi'} T^+_{\pi'}.$$

\begin{pro}[Uniform hyperbolicity] For each $k$ there exists $\mu=\mu(k) > 1$ and $C_1, C_2> 0$ with the following property: Let $(\pi^n,\ve^n)$ be  a sequence de combinatorics k-bounded with $r_{\ve^n}(\pi^n)=\pi^{n+1}.$ Then
\label{hyp}
\begin{itemize}
\item[(a)] For every $n$ and $v \in C^u_{\pi^0}$ we have
    $$  \|(\Theta_{\pi^n,\ve^n} \cdots \Theta_{\pi^1,\ve^1} \Theta_{\pi^0,\ve^0}) v\| \geq C_1 \mu^n \|v\|.$$
\item[(b)] For every $n$ and $v \in C^s_{\pi^n}$ we have
    $$  \|(\Theta_{\pi^{n-1},\ve^{n-1}} \cdots \Theta_{\pi^1,\ve^1} \Theta_{\pi^0,\ve^0})^{-1} v\| \geq C_2 \mu^n \|v\|.$$
\end{itemize}
\end{pro}
\begin{proof} Note that for every $n$ the finite sequence $$\{(\pi^n,\ve^n), (\pi^{n+1},\ve^{n+1}), \dots, (\pi^{n+k},\ve^{n+k})\}$$ is {\it complete}, that is,  every letter $\alpha \in \mathcal{A}$ is the winner  at least once along this sequence.  It follows  from \cite[Section 1.2.4]{coho} and \cite[Section 10.3]{yoc} that
$$ \Theta_{n,n+k(3d-4)}=\Theta_{\pi^{n+k(3d-4)},\ve^{n+k(3d-4)}} \cdots \Theta_{\pi^{n+1},\ve^{n+1}} \Theta_{\pi^n,\ve^n}$$
is a positive matrix with integer entries  satisfying
\begin{eqnarray}
\label{tmais}
{}^t(\Theta_{n,n+k(3d-4)})^{-1}\overline{T^+_{\pi^n}}\subset T^+_{\pi^{n+k(3d-4)+1}}\cup\{0\}.
\end{eqnarray}
By  (\ref{tmais}) and  $\Theta_{\pi^j,\ve^j}\overline{\Om_{\pi^j}(T^+_{\pi^j})}=\Om_{\pi^{j+1}}{}^t\Theta_{\pi^j,\ve^j}^{-1}(\overline{T^+_{\pi^j}})$ for all $j\geq0$ we have
\begin{equation} \Theta_{n,n+k(3d-4)}\overline{C^u_{\pi^n}} \subset C^u_{\pi^{n+k(3d-4)+1}}\cup\{0\}\end{equation}
for every $n \in \mathbb{N}$. Since $C^u_{\pi^n}\subset \mathbb{R}^{\mathcal{A}}_+$, in particular we have
\begin{equation}\label{posi}  \|\Theta_{n,n+k(3d-4)} v\|_1\geq d \|v\|_1,\end{equation}
for every $v \in C^u_{\pi^n}$. Given $n \in \mathbb{N}$, let $n= qk(3d-4) + r$, with $q,r \in \mathbb{N}$, $0\leq r < k(3d-4)$. Then
$$\|\Theta_{0,n} v\|\geq d^{q} \|\Theta_{0,r} v\|_1\geq d^{(n-r)/(k(3d-4)} \min \{\|\Theta_{0,r}^{-1}\|^{-1}, \ r< k(3d-4) \}  \|v\|_1$$
$$= C_1 \mu^n \|v\|_1 $$
To show  (b),  note that by (a) we have that for every $n$
$${}^t\Theta_{n+k(3d-4),n}= {}^t( \Theta_{\pi^n,\ve^n} \Theta_{\pi^{n+1},\ve^{n+1}} \cdots  \Theta_{\pi^{n+k(3d-4)},\ve^{n+k(3d-4)}}).$$
has positive  integer entries. Using an argument similar to the proof of (a) we conclude that
$$ \|{}^t( \Theta_{\pi^n,\ve^n}\cdots \Theta_{\pi^0,\ve^0} ) w\| \geq C_1 \mu^n \|w\|$$
for every $w \in \mathbb{R}^\mathcal{A}_+$.  By \eqref{semiconj} we have
%
$$(\Theta_{\pi^{n-1},\ve^{n-1}}\cdots\Theta_{\pi^0,\ve^0})^{-1}\Om_{\pi^n}=\Om_{\pi^0}{}^t(\Theta_{\pi^{n-1},\ve^{n-1}}\cdots\Theta_{\pi^0,\ve^0}).$$
Given  $ v \in C^s_{\pi^0}$ there exists $w \in \mathbb{R}_+^{\mathcal{A}}$ such that $v=\Omega_{\pi^0}w$.
The fact that $\Theta_{i+k(3d-4),i}^t >0$ for every $i$ easily implies that there exist $\delta_1,\delta_2  \in (0,1)$ such that
$${}^t\Theta_{0,n} \mathbb{R}_+^\mathcal{A} \subset \Lambda_{\delta_1,\delta_2} =\{(\lambda_\alpha)_{\alpha \in \mathcal{A}} \in \mathbb{R}_+^\mathcal{A}, \ \delta_1\leq \frac{\lambda_\alpha }{\sum_\beta \lambda_\beta} \leq \delta_2   \text{ for every } \alpha \}.$$

Now note that
\begin{equation}\label{kerzero} \mathbb{R}_+^\mathcal{A} \cap Ker \ \Omega_{\pi^n}=\emptyset.\end{equation}

In fact, let $\lambda \in \mathbb{R}_+^\mathcal{A}$ be  such that  $\Omega_{\pi^n}\lambda=0.$ Then by definition of $\Om_{\pi^n}$ we have

$$\sum_{\pi_1^n(\beta)<\pi_1^n(\alpha)} \lambda_{\beta}-\sum_{\pi_0^n(\beta)<\pi_0^n(\alpha)}\lambda_{\beta}=0,\; \text{ for all }\al\in\A.$$

Let $\al_0\in\A$ such that $\pi_0^n(\al_0)=d.$ Then

\begin{eqnarray}
\label{soma}
0&=&\sum_{\pi_1^n(\beta)<\pi_1^n(\alpha_0)} \lambda_{\beta}-\sum_{\pi_0^n(\beta)<d}\lambda_{\beta}\nonumber\\
&=& \sum_{\pi_1^n(\beta)<\pi_1^n(\alpha_0)} \lambda_{\beta} - \sum_{\pi_1^n(\beta)<\pi_1^n(\alpha_0)} \lambda_{\beta}-\sum_{\pi_1^n(\beta)>\pi_1^n(\alpha_0)} \lambda_{\beta}\nonumber\\
& = & -\sum_{\pi_1^n(\beta)>\pi_1^n(\alpha_0)} \lambda_{\beta},
\end{eqnarray}
which is a contradiction because $\sharp\{\beta\in\A\; :\; \pi_1(\beta)>\pi_1(\alpha_0)\}\geq1,$ due to the fact that $\pi^n=(\pi_0^n,\pi_1^n)$ is irreducible.

By \ref{kerzero} we have that

$$C_3=\inf \{\frac{\|\Omega_{\pi}u\|_1}{\|u\|_1}, \ u\neq 0, u \in \Lambda_{\delta_1,\delta_2}, \ \pi \in \Pi^1 \ and   \ irreducible \} >0.$$

For all $v\in C^s_{\pi^n}$ there is $w\in\R^\A_+$ such that $v=\Om_{\pi^n}w.$ Therefore

\begin{eqnarray*}
\|(\Theta_{\pi^{n-1},\ve^{n-1}}\cdots\Theta_{\pi^0,\ve^0})^{-1}v\|&=&\|\Om_{\pi^0}{}^t(\Theta_{\pi^{n-1},\ve^{n-1}}\cdots\Theta_{\pi^0,\ve^0})w\|\\
&\geq& C_3\cdot \|{}^t(\Theta_{\pi^{n-1},\ve^{n-1}}\cdots\Theta_{\pi^0,\ve^0})w\|\\
&\geq & C_3\cdot C_1\cdot \mu^{n-1}\|w\|\\
&\geq & C_3\cdot C_1\cdot\frac{1}{C_4}\cdot\mu^{n-1}\|v\|,
\end{eqnarray*}
where
$$C_4=\sup_{\pi\in\Pi^1}\sup_{v\in\R^\A\setminus\{0\}}\left\{\frac{\|\Om_{\pi}v\|}{\|v\|}\right\}.$$

\end{proof}

Motivated by Proposition \ref{hyp} we define the stable direction in the point $\{\pi^j,\ve^j\}$ as
\begin{eqnarray}
\label{stable}
E^s_j:= E^s(\pi^j)=\bigcap_{n\geq0}\Theta_j^{-1}\cdots\Theta_{j+n}^{-1}(C^s_{\pi^{j+n+1}}).
\end{eqnarray}

By definition the subspaces $E^s_j$ are invariant by the Rauzy-Veech cocycle, i.e, for all $j\geq0$
$$\Theta_j(E^s_j)=E^s_{j+1}.$$

Now we define the unstable direction. Let $u_0\in C^u_{\pi^0}$ be such that $\|u_0\|=1.$ Then we define $E^u_0$ as the subspace spanned by $u_0,$ that we will be denoted by $<u_0>.$ For all $j>0$ we define
\begin{eqnarray}
\label{unstable}
E^u_j:=<\frac{u_j}{\|u_j\|}>, \text{ where } u_j=\Theta_{j-1}(u_{j-1}).
\end{eqnarray}
The subspaces $E^u_j$ are forward invariant  by the Rauzy-Veech cocycle.

The result of  this subsection shows that Rauzy-Veech cocycle is hyperbolic inside $\im\ \Om.$ In the next subsection we show that outside $\im\ \Om$ the Rauzy-Veech cocycle has a central direction and it is a quasi-isometry in this direction.

\subsection{Central direction: Periodic combinatorics}
First we study periodic combinatorics. Suppose that there is $p\in\N$ such that $\{\pi^n,\ve^n\}=\{\pi^{n+p},\ve^{n+p}\}$ for all $n\in\N$, i.e. the combinatorics has  period $p.$ So we know that $(\Theta^t_{0,p-1})^{-1}|_{\ker\ \Omega_{\pi^0}}=\mathrm{Id},$ see \cite[Lemma 2.11]{viana}.

\begin{lemma}  \label{Ec}  Define
$\Psi_p\colon  \ker \ \Omega_{\pi^0} \rightarrow \im \ \Omega_{\pi^0}$
 as
$$\Psi_p(k)=(\Theta_{0,p-1}-\mathrm{Id})^{-1}(k- \Theta_{0,p-1}(k)).$$ Then the subspace
$E^c_{0,p-1}:=\Big\{k+\Psi_p(k), \ k\in \ker \ \Omega_{\pi^0}\Big\}$ is  the central direction  of  $\Theta_{0,p-1}.$ Indeed  $\Theta_{0,p-1}v=v$ for every $v\in E^c_{0,p-1}$.\end{lemma}

\dem Since $\Theta_{0,p-1}-\mathrm{Id}$ is not invertible on $\mathbb{R}^\mathcal{A}$, firstly we  show that $\Psi_p$ is well defined. We claim  that $k- \Theta_{0,p-1}(k)\in \im \ \Omega_{\pi^0}.$ Indeed, using the fact that $(\Theta^t_{0,p-1})^{-1}|_{\ker\ \Omega_{\pi^0}}=\mathrm{Id},$ we have that for all $u\in \ker \ \Omega_{\pi^0}$

\begin{eqnarray*}
<u, k- \Theta_{0,p-1}(k)>&=&<u, k>-< \Theta^t_{0,p-1}(u), k>\\
& =& <u, k>-<u, k>=0.
\end{eqnarray*}
Therefore $k- \Theta_{0,p-1}(k)\perp \ker \ \Omega_{\pi^0},$ which proves our claim. By Proposition \ref{hyp} we have that $\Theta_{0,p-1}$ is hyperbolic in $ \im \ \Omega_{\pi^0}$, so  $$\Theta_{0,p-1}-\mathrm{Id}\colon \im \ \Omega_{\pi^0} \rightarrow \im \ \Omega_{\pi^0} $$ is invertible on  $\im \ \Omega_{\pi^0}$ and  we can define
$$\Psi_p(k):=(\Theta_{0,p-1}-\mathrm{Id})^{-1}(k- \Theta_{0,p-1}(k)).$$
We claim that  $\Theta_{0,p-1}v=v$ for every $v\in E^c_{0,p-1}$. Indeed by the definition of $\Psi_p$
$$(\Theta_{0,p-1}-Id)(k+\Psi_p(k))=0.$$
Note that  $\mathrm{dim}\ E^c_{0,p-1}= \mathrm{dim}\ \ker \Om_{\pi^0}=d-2.$ So $E^c_{0,p-1}$ is the central direction.\cqd
The  next result shows the invariance of $E^c_{0,p-1}$ by the Rauzy-Veech cocycle.
\begin{lemma}
\label{inv}  $\Theta_{\pi^0}(E^c_{0,p-1})=E^c_{1,p}.$
\end{lemma}

\dem
Let $v\in E^c_{0,p-1}.$ Then
$$\Theta_{\pi^{p-1}}\cdots \Theta_{\pi^0}(v)=v.$$
Applying $\Theta_{\pi^0}$ to both sides
$$\Theta_{\pi^0}\cdot \Theta_{\pi^{p-1}}\cdots \Theta_{\pi^1}\cdot\Theta_{\pi^0}(v)=\Theta_{\pi^0}(v).$$
So
$$\Theta_{1,p}\cdot\Theta_{\pi^0}(v)=\Theta_{\pi^0}(v) \Rightarrow \Theta_{\pi^0}(v)\in E^c_{1,p}.$$

\cqd

We now prove that the Kontsevich-Zorich cocycle behaves as a quasi-isometry in its central direction. By Proposition \ref{hyp} we can choose $n_0>0$ and $\mu >>1$ such that
$$\|\Theta_{n,n+n_0}(x)\|\geq \mu \|x\| \;\forall x\in C^u_{\pi^n}\quad \text{and}\quad\|\Theta_{n+n_0,n}(x)\|\geq \mu \|x\| \;\forall x\in C^s_{\pi^n}.$$

For $\epsilon>0,$ define the cones  $C^n_{\epsilon,u}$ and $C^n_{\epsilon,s}$, where $C^n_{\epsilon,u}$ is the set of vectors $x=x_k+x_i\in \ker \Omega_{\pi^n}\oplus \im \Omega_{\pi^n}$ such that
\begin{itemize}
\item $\|x_k\|\leq \epsilon \|x_i\|$,\\
\item We have that $x_i= x_i^s + x_i^u$, where $ x_i^s\in\Theta_{n+n_0-1,n}C^s_{\pi^{n+n_0}}\subset C^s_{\pi^n}$,  $x_i^u\in\Theta_{n-n_0,n-1}C^u_{\pi^{n-n_0}}\subset C^u_{\pi^n}$ and  $\|x_i^s\|\leq\|x_i^u\|.$
\end{itemize}
and  we define analogously $C^n_{\epsilon,s}$ replacing  the last condition by $\|x_i^u\|\leq\|x_i^s\|.$ Define also
 $C^n_{\epsilon}:=C^n_{\epsilon,u}\cup C^n_{\epsilon,s}$.
\begin{pro}
\label{inv.cone}
There exists $\epsilon_0=\epsilon_0(k)>0$ and $\gamma <1$ such that if $\epsilon<\epsilon_0$ then
$$\Theta_{n,n+n_0-1}(C^n_{\epsilon,u})\subset C^{n+n_0}_{\gamma\epsilon,u}\;\text{ and }\; \Theta_{n-1,n-n_0}(C^n_{\epsilon,s})\subset C^{n-n_0}_{\gamma\epsilon,s}.$$
\end{pro}
Before proving the Proposition \ref{inv.cone} we need some lemmas.
\begin{lemma} \label{25} 
%
There exists $C_{5} >0$ such that for   $n_0$ large enough and  all $m > n_0$ there are linear projections $\Pi_{m}^s$ and $\Pi_{m}^u$ defined in $Im\ \Omega_{\pi^{m}}$    such that for every $v\in \ Im\ \Omega_{\pi^{m}}$ we have that  $v_s=\Pi_{m}^s(v)\in \pm \Theta_{m+n_0-1,m}(C^s_{\pi^{m+n_0}})$ and  $v_u=\Pi_{m}^u(v) \in \pm \Theta_{m-n_0,m-1}C^u_{\pi^{m-n_0}}$ satisfy $v= v_s+v_u$ and
\begin{equation} \label{estcoones} \|v_s\|, \| v_u \| \leq C_{5} \|v\|.\end{equation}
\end{lemma}
\dem For a fixed $n_0$ there exists only a finite number  of matrices $\Theta_{m-n_0,m-1}$ and  $\Theta_{m+n_0-1,m}$. The same holds for the subspaces $Im\ \Omega_{\pi^{m}}$ and cones $C^u_{\pi^{m-n_0}}$, $C^s_{\pi^{m+n_0}}$,  $\Theta_{m+n_0-1,m}(C^s_{\pi^{m+n_0}})$ and  $\Theta_{m-n_0,m-1}C^u_{\pi^{m-n_0}}$. Moreover
\begin{equation} \label{finitecond}\pm \overline{\Theta_{m+n_0-1,m}(C^s_{\pi^{m+n_0}})}\cap \pm \overline{\Theta_{m-n_0,m-1}C^u_{\pi^{m-n_0}}} =\{ 0\}.\end{equation}
For each possible combination of matrices, cones e subspaces, choose $w_s \in \pm \Theta_{m+n_0-1,m}(C^s_{\pi^{m+n_0}})$ and $w_u \in \pm \Theta_{m-n_0,m-1}C^u_{\pi^{m-n_0}}$. Then
$v = c_s w_s + c_u w_u$, with $c_s, c_u \in \mathbb{R}$. Define $\Pi_{m}^s(v)=c_s w_s$ and $\Pi_{m}^u(v)=c_u w_u$. Let $C_5$ be  the supremum of the norms of all projections $\Pi_{m}^s$, $\Pi_{m}^u$ over all possible combinations of matrices, cones, spaces and choices of $w_s$ and $w_u$.  This supremum is finite due (\ref{finitecond}). Finally, note that the same $C_5$ satisfies (\ref{estcoones}) if we replace $n_0$ by some $n_1\geq n_0$, because $\Theta_{m+n_1-1,m}(C^s_{\pi^{m+n_1}}) \subset \Theta_{m+n_0-1,m}(C^s_{\pi^{m+n_0}})$ , $\Theta_{m-n_1,m-1}C^u_{\pi^{m-n_1}} \subset \Theta_{m-n_0,m-1}C^u_{\pi^{m-n_0}}$ and the freedom to choose $w_s \in \pm \Theta_{m+n_0-1,m}(C^s_{\pi^{m+n_0}})$ and $w_u \in \pm \Theta_{m-n_0,m-1}C^u_{\pi^{m-n_0}}$ as we like.  
\cqd
\begin{lemma}
\label{iso.ker}
For all $n>n_0$ and for all $x_k\in \ker\Omega_{\pi^n}$
\begin{eqnarray*}
\|\Theta_{n,n+n_0-1}(x_k)\|_{\ker\Omega_{\pi^{n+n_0}}}&:=& \sup_{k_{n+n_0}\in \ker\Omega_{\pi^{n+n_0}}}<\Theta_{n,n+n_0-1}(x_k),k_{n+n_0}>\\
&=& \|x_k\|_{\ker\Omega_{\pi^{n}}},
\end{eqnarray*}
where $<\cdot,\cdot>$ denote the usual inner product of $\mathbb{R}^\A.$
\end{lemma}
\dem
Note that ${}^t\Theta_{n,n+n_0-1}(\ker \Omega_{\pi^{n+n_0}})=\ker\Omega_{\pi^n}.$ Therefore
\begin{eqnarray*}
\lefteqn{\sup_{k_{n+n_0}\in \ker\Omega_{\pi^{n+n_0}}}<\Theta_{n,n+n_0-1}(x_k),k_{n+n_0}>=}\\
& & =  \sup_{k_{n+n_0}\in \ker\Omega_{\pi^{n+n_0}}}<x_k,{}^t\Theta_{n,n+n_0-1}(k_{n+n_0})>\\
& & =  \sup_{k_{n}\in \ker\Omega_{\pi^{n}}}<x_k,k_{n}>\\
& & = \|x_k\|_{\ker\Omega_{\pi^{n}}}.
\end{eqnarray*}
\cqd
\begin{proof}[Proof of Proposition \ref{inv.cone}:]
Let $x=x_i+x_k\in C^n_{\epsilon,u}.$ First note that
\begin{eqnarray}
\label{xi}
\|\Theta_{n,n+n_0-1}(x_i)\| &=&\|\Theta_{n,n+n_0-1}(x_i^u+x_i^s)\|\nonumber\\
& \geq & \mu\|x_i^u\|-\frac{1}{\mu}\|x_i^s\|\nonumber\\
& \geq & \left(\mu-\frac{1}{\mu}\right)\|x_i^u\|,
\end{eqnarray}
and that
\begin{eqnarray}
\label{im}
\|\Theta_{n,n+n_0-1}(x_k)\|_{\im\Omega_{\pi^{n+n_0}}}&\leq& \|\Theta_{n,n+n_0-1}(x_k)\|\nonumber\\
&\leq&\|\Theta_{n,n+n_0-1}\|\cdot \|x_k\|\nonumber\\
&\leq &\epsilon\cdot\|\Theta_{n,n+n_0-1}\|\cdot\|x_i\|.
\end{eqnarray}
Note that by Lemma \ref{iso.ker} we have
\begin{eqnarray}
\label{new}
\|\Theta_{n,n+n_0-1}(x_k)\|_{\ker\ \Omega_{\pi^{n+n_0}}}&=& \|x_k\|_{\ker\ \Omega_{\pi^{n}}}\leq \epsilon\|x_i\|\nonumber\\
& \leq & 2\epsilon\|x_i^u\|.
\end{eqnarray}
Then
\begin{eqnarray}
\lefteqn{\|\Theta_{n,n+n_0-1}(x)\|_{\im\Omega_{\pi^{n+n_0}}}}\nonumber\\
 & & \geq  \|\Theta_{n,n+n_0-1}(x_i)\|_{\im\Omega_{\pi^{n+n_0}}}-\|\Theta_{n,n+n_0-1}(x_k)\|_{\im\Omega_{\pi^{n+n_0}}}\nonumber\\
& &\geq  \left(\mu-\frac{1}{\mu}\right)\|x_i^u\| - 2\epsilon\cdot\|\Theta_{n,n+n_0-1}\|\cdot\|x_i^u\|, \text{ by } \eqref{xi}\nonumber\\
& & \geq  \frac{1}{2\epsilon}\left(\mu-\frac{1}{\mu}-2\epsilon\cdot\|\Theta_{n,n+n_0-1}\|\right) \|\Theta_{n,n+n_0-1}(x)\|_{\ker\Omega_{\pi^{n+n_0}}}\nonumber\\
& & \geq  \frac{1}{2\epsilon}\left(\mu-\frac{1}{\mu}-2\epsilon\cdot C\right) \|\Theta_{n,n+n_0-1}(x)\|_{\ker\Omega_{\pi^{n+n_0}}}
\end{eqnarray}
Here $C$ depends only on $n_0$. Therefore
\begin{eqnarray}
\lefteqn{\|\Theta_{n,n+n_0-1}(x)\|_{\ker\Omega_{\pi^{n+n_0}}}}\nonumber\\
& & \leq 2\epsilon\left(\mu-\frac{1}{\mu}-2\epsilon\cdot C\right)^{-1} \|\Theta_{n,n+n_0-1}(x)\|_{\im\Omega_{\pi^{n+n_0}}}.
\end{eqnarray}
Choose $\epsilon_0$  small enough such that

$$\gamma:= 2\epsilon_0\left(\mu-\frac{1}{\mu}-2\epsilon_0\cdot C\right)^{-1} < 1.$$

Note that
$$\Theta_{n,n+n_0-1}(x_i^u)\in\Theta_{n,n+n_0-1}C^u_{\pi^{n}}$$
and
$$v=\Theta_{n,n+n_0-1}(x_i^s)\in  C^s_{\pi^{n+n_0}} $$
Let $v_s$ and $v_u$, $v=v_s+v_u$, be as in  Lemma \ref{25}. Note that
$$\|v_s\|, \|v_u\|\leq C_5 \|v\|\leq \frac{C_5}{\mu} \| x_i^s  \|\leq \frac{C_5}{\mu} \| x_i^u  \| \leq \frac{C_5}{\mu^2} \|  \Theta_{n,n+n_0-1}(x_i^u)\ \| $$
Then $v_u + \Theta_{n,n+n_0-1}(x_i^u) \in \Theta_{n,n+n_0-1}C^u_{\pi^{n}}$ and $v_s \in \Theta_{n+2n_0-1,n+n_0}(C^s_{\pi^{n+2n_0}})$ and moreover
$$\|v_u + \Theta_{n,n+n_0-1}(x_i^u)\|\geq  (1-\frac{C_5}{\mu^2})\|  \Theta_{n,n+n_0-1}(x_i^u)\geq \frac{\mu^2}{C_5}(1-\frac{C_5}{\mu^2})\|  v_s\|\geq \|v_s\|,$$
so $\Theta_{n,n+n_0-1}(x)\in C^{n+n_0}_{\gamma\epsilon,u}.$ The proof of $\Theta_{n-1,n-n_0}(C^n_{\epsilon,s})\subset C^{n-n_0}_{\gamma\epsilon,s}$ is analogous.

\end{proof}
\begin{pro}
\label{central.bounded}
Suppose that $p>n_0.$ Then
$$\sup_{{k\in \ker\Omega_{\pi^0}}\atop{k\not=0}}\frac{\|\Psi_p(k)\|}{\|k\|}\leq \frac{1}{\epsilon_0}.$$
\end{pro}
\dem
Suppose by contradiction that this claim is false. Then we could find $\epsilon<\epsilon_0$  such that
\begin{eqnarray}
\label{bounded}
\sup_{{k\in \ker\Omega_{\pi^0}}\atop{k\not=0}}\frac{\|\Psi_p(k)\|}{\|k\|}=\frac{1}{\epsilon}> \frac{1}{\epsilon_0}.
\end{eqnarray}
Let $k_0\in \ker\Omega_{\pi^0}$ be such that the supremum above is attained on it. Thus $k_0+\Psi_p(k_0)\in C^0_{\epsilon}:=C^0_{\epsilon,u}\;\dot{\cup}\;C^0_{\epsilon,s}.$ Assume, without loss of generality,  that $k_0+\Psi_p(k_0)\in C^0_{\epsilon,u}.$ By assumption  $E^c_{0,p-1}=E^c_{n_0p,(n_0+1)p-1}$ and by Lemma \ref{inv}  we have $\Theta_{0,n_0p-1}(E^c_{0,p-1})=E^c_{n_0p,(n_0+1)p-1}$ . Thus by Proposition \ref{inv.cone}
$$\tilde{k}_0+\Psi_p(\tilde{k}_0)=\Theta_{0,n_0p-1}(k_0+\Psi_p(k_0))\in C^{n_0p}_{\gamma\epsilon,u}\Longleftrightarrow$$
$$\|\tilde{k}_0\|\leq\gamma\epsilon \|\Psi_p(\tilde{k}_0)\|\Longleftrightarrow \frac{ \|\Psi_p(\tilde{k}_0)\|}{\|\tilde{k}_0\|}\geq \frac{1}{\gamma\epsilon}>\frac{1}{\epsilon},$$
which contradicts \eqref{bounded}.
\cqd
\subsection{Central direction: arbitrary $k$-bounded  combinatorics} Let $f\in\mathcal{B}_{k}^{2+\nu}$ and $\gamma(f)=\{\pi^i,\ve^i\}_{i\in\N}$ be  its combinatorics. For each $n\in\N$ we define the new periodic combinatorics, that will be denoted by  $\gamma_n(f)=\{\tilde{\pi}^i,\tilde{\ve}^i\}_{i\in\N}$:
\begin{itemize}
\item[(a)] For $i\leq n$ define $(\tilde{\pi}^i,\tilde{\ve}^i)=(\pi^i,\ve^i),$ and denote  $\tilde{\gamma}_n=\{\tilde{\pi}^i,\tilde{\ve}^i\}_{i=0}^n;$
\item[(b)] Let $\tilde{\gamma}_{n,p_n}=\{(\tilde{\pi}^n,\tilde{\ve}^n),...,(\tilde{\pi}^{p_n},\tilde{\ve}^{p_n})\}$ be an admissible sequence of combinatorics, i.e., $r_{\tilde{\ve}^i}(\tilde{\pi}^i)=\tilde{\pi}^{i+1}$ for all $n\leq i <p_n$ with $(\tilde{\pi}^{p_n},\tilde{\ve}^{p_n})=(\pi^0,\ve^0).$ It is possible to get this sequence by \cite{viana}.
\end{itemize}
Then define $\gamma_n(f)=(\tilde{\gamma}_n\ast\tilde{\gamma}_{n,p_n})\ast(\tilde{\gamma}_n\ast\tilde{\gamma}_{n,p_n})\ast\cdots.$ Note that the combinatorics $\gamma_n(f)$ is periodic of period $p_n$ and that $\gamma_n(f)\to \gamma(f)$ when $n\to\infty.$ The Rauzy-Veech cocycle associetade to $\gamma_n(f)$ will be denoted by $\tilde{\Theta}.$
By Lemmas \ref{Ec} and \ref{inv} we have that for all $s\geq0$ the subspace $E^c_{s,p_n}$ is the graph of $\Psi_{s,p_n}$ and $\tilde{\Theta}_s(E^c_{s,p_n})=E^c_{s+1,p_n}.$ By Proposition \ref{central.bounded} the sequence $\{\Psi_{0,p_n}\}_{n\in\N}$ is equicontinuous and uniformly bounded, so it admits a  subsequence $\{\Psi_{0,p_n}\}_{n\in\N_0}$ that uniformly converges. The same holds for  $\{\Psi_{1,p_n}\}_{n\in\N_0},$ i.e. we can find a infinite subset $\N_1\subset\N_0$ such that $\{\Psi_{1,p_n}\}_{n\in\N_1}$ is uniformly converge. Proceeding analogously for each  $j\in\N$ we can find a infinite subset  $\N_j\subset\N,$ such that $\N_0\supset\N_1\supset\cdots\supset\N_j\supset\cdots$ and $\{\Psi_{j,p_n}\}_{n\in\N_j}$ uniformly converge. Now define the infinite set $\tilde{\N}\subset \N$ taking as your $j$-th element the $j$-th element of $\N_j.$ Define the subspace $E^c_{j,\infty}$ as the graph of
$$\Psi_{j,\infty}=\lim_{n\to\infty\atop{n\in\tilde{\N}}}\Psi_{j,p_n}.$$
\index{$\Psi_{j,\infty}$}
By construction we have that $\dim E^c_{j,\infty}=d-2.$ The next easy Lemma show that the subspaces $E^c_{j,\infty}$ are invariant by the Rauzy-Veech cocycle.
\begin{lemma}
For all $j\geq0,$ we have that $\Theta_j(E^c_{j,\infty})=E^c_{j+1,\infty}.$
\end{lemma}
\dem
Denote by $n_j\in\N$ the $j-$th element of $\tilde{\N}.$
\begin{eqnarray*}
\Theta_j(E^c_{j,\infty})&=&\Theta_j (\lim_{n\to\infty\atop{n\in\tilde{\N}}}\mathrm{graph}(\Psi_{j,p_n}))\\
& = &\lim_{n\to\infty\atop{n\geq n_j}}\Theta_j( \mathrm{graph}(\Psi_{j,p_n}))\\
& = & \lim_{n\to\infty\atop{n\geq n_j}}  \mathrm{graph}(\Psi_{j+1,p_n}), \text{ by Lemma \ref{inv}}\\
& = & E^c_{j+1,\infty}.
\end{eqnarray*}
\cqd
\begin{pro}[Quasi-isometry in the central direction]
\label{quasi.iso}
For all vector $v\in E^c_{0,\infty}$ and for all $n\geq0$, there is  $C_{6}>1$ such that
$$\frac{1}{C_{6}}\cdot\|v\|\leq \|\Theta_{0,n}v\| \leq C_{6}\cdot \|v\|,$$
where $\Theta_{0,n}=\Theta_n\Theta_{n-1}\cdots\Theta_0.$
\end{pro}

\dem
For all $v\in E^c_{0,\infty},$ there is $k\in\ker\ \Omega_{\pi^0}$ such that $v=k+\Psi_{0,\infty}(k).$ By continuity of inner product we have that $k\perp\Psi_{0,\infty}(k).$ Then
$$\|k\|\leq \sqrt{\|k\|^2+\|\Psi_{0,\infty}(k)\|^2}=\|v\|\leq \|k\|+\|\Psi_{0,\infty}(k)\|\leq \|k\|+\frac{1}{\epsilon_0}\|k\|,$$
by Proposition \ref{central.bounded}. Therefore
\begin{eqnarray}
\label{Lc1}
\|v\|_{ker \ \Omega_{\pi^0}}:=\|k\|\leq \|v\|\leq \left(1+\frac{1}{\epsilon_0}\right)\cdot \|k\|.
\end{eqnarray}
 Let $\mathcal{C}:=\{u=(u_\al)_{\al\in\A}\in\R^{\A} : u_\al=-1,0,\text{ or }1\}.$ Then
$$\max_{u\in\mathcal{C}}\|u\|\leq \sqrt{d}=:C_{7}.$$
We know  that ${}^t\Theta_{0,n}:\ker\ \Omega_{\pi^n}\to\ker\ \Omega_{\pi^0}$  maps the  basis of $\ker\ \Omega_{\pi^n}$ to  the basis of $\ker\ \Omega_{\pi^0},$ by \cite[Lemma 2.16]{viana}.

For all $\pi\in\Pi^g$ and for all $k\in\ker\ \Omega_{\pi}$ we can put $k=\sum_{i=1}^da_iu_i,$ where $a_i\in\R$ and $u_i$ belongs to the basis of $\ker\ \Omega_{\pi}$, in particular  $u_i \in\mathcal{C}$ (see \cite{viana}). Using the sum norm we have that

\begin{eqnarray*}
\|{}^t\Theta_{0,n}k\|&=&\|\sum_{i=1}^d a_i{}^t\Theta_{0,n}u_i\|\\
& = & \|\sum_{i=1}^d a_i\tilde{u}_i\|, \text{ where } \tilde{u}_i\in\mathcal{C}\\
&\leq & \sum_{i=1}^d |a_i|\cdot\|\tilde{u}_i\|\\
&\leq& C_{7}\cdot \|k\|.
\end{eqnarray*}
Therefore
\begin{eqnarray}
\label{c11}
\|{}^t\Theta_{0,n}\|_{\ker\ \Omega_{\pi^n}}\leq C_{7}.
\end{eqnarray}
Then
\begin{eqnarray}
\label{Lc2}
\|\Theta_{0,n}v\|_{\ker\ \Omega_{\pi^n}}&=&\sup_{k\in\ker\ \Omega_{\pi^n}\atop{\|k\|\leq1}}<\Theta_{0,n}v,k>\nonumber\\
&=& \sup_{k\in\ker\ \Omega_{\pi^n}\atop{\|k\|\leq1}}<v,{}^t\Theta_{0,n}k>\nonumber\\
&\leq&\|v\|\cdot \|{}^t\Theta_{0,n}\|_{\ker\ \Omega_{\pi^n}}\|k\|\nonumber\\
&\leq& C_{7}\cdot \|v\|, \text{ due }\eqref{c11}.
\end{eqnarray}
Now
\begin{eqnarray}
\label{menor}
\|\Theta_{0,n}v\|&\leq & \left(1+\frac{1}{\epsilon_0}\right)\|\Theta_{0,n}v\|_{\ker\ \Omega_{\pi^n}}, \text{ due } \eqref{Lc1}\nonumber\\
&\leq& \left(1+\frac{1}{\epsilon_0}\right)\cdot C_{7}\cdot \|v\|\text{, due } \eqref{Lc2}.
\end{eqnarray}

Now we can find the lower estimate for  $\|\Theta_{0,n}v\|.$  First note that

\begin{eqnarray}
\label{maior}
\|\Theta_{0,n}v\|_{\ker\ \Omega_{\pi^n}}&=&\sup_{k\in\ker\ \Omega_{\pi^n}\atop{\|k\|\leq1}}<v,{}^t\Theta_{0,n}k>\nonumber\\
&\geq & \sup_{\tilde{k}\in\ker\ \Omega_{\pi^0}\atop{\|\tilde{k}\|\leq\frac{1}{C_7}}}<v,\tilde{k}>\nonumber\\
&=&\frac{1}{C_{7}}\cdot \sup_{\tilde{k}\in\ker\ \Omega_{\pi^0}\atop{\|\tilde{k}\|\leq1}}<v,\tilde{k}>\nonumber\\
&=&\frac{1}{C_{7}}\cdot\|v\|_{\ker\ \Omega_{\pi^0}}.
\end{eqnarray}

Using \eqref{Lc1} we have that
$$\|v\|\leq \left(1+\frac{1}{\ve_0}\right)\cdot \|v\|_{\ker\ \Omega_{\pi^0}} \Rightarrow \left(1+\frac{1}{\ve_0}\right)^{-1}\|v\|\leq \|v\|_{\ker\ \Omega_{\pi^0}}.$$

This estimate jointly with \eqref{maior} and \eqref{menor}, yields

$$\frac{1}{C_{7}}\left(1+\frac{1}{\epsilon_0}\right)^{-1}\|v\|\leq \frac{1}{C_{7}}\|v\|_{\ker\ \Omega_{\pi^0}}\leq \|\Theta_{0,n}v\|_{\ker\ \Omega_{\pi^n}}\leq \|\Theta_{0,n}v\| \leq C_{7}\left(1+\frac{1}{\epsilon_0}\right)\|v\|.$$

Taking $C_{6}= C_{7}\left(1+\frac{1}{\epsilon_0}\right)$ we have the result.
\cqd

\section{Proof of Theorems \ref{teo1} and \ref{teo2}}


Let $f:[0,1)\to[0,1)$ be a g.i.e.m.. For simplicity we write $R^n(f)(x)=f_n(x)=f^{q_n^{\al}}(x)$ if $x\in I^n_\al.$ Define  $L^n=(L_{\al}^n)_{\al\in\A}$ by

\begin{eqnarray}
\label{mean-deriv}
L_{\al}^n=\frac{1}{|I^n_{\al}|}\int_{I^n_{\al}}\ln Df_n(s)ds=\frac{1}{|I^n_{\al}|}\int_{I^n_{\al}}\ln Df^{q^{\al}_n}(s)ds.
\end{eqnarray}
Note that if $f$ is a affine i.e.m. then $L_{\al}^n=\omega^n_{\al}$ for all $\al\in\A.$ The following proposition gives a relationship between $L^n$ and $L^{n+1},$ more precisely we prove that  $L^n$ is an asymptotic pseudo-orbit for the Kontsevich-Zorich cocycle.

\begin{pro}
\label{L}
Let $f\in\mathcal{B}_{k}^{2+\nu}$ with $\int_0^1D^2(f)(s)/Df(s)ds=0.$ Then
\begin{eqnarray}
\label{formula}
L^{n+1}=\Theta_nL^n+\vec{\epsilon_n},
\end{eqnarray}
where $\|\vec{\epsilon_n}\|=\O(\la^{\sqrt{{n}}}),$ $0<\la<1.$
\end{pro}

\dem Denote by $x^n_{\al}\in I^n_{\al}$ the point such that $L_{\al}^n=\ln Df^{q^n_{\al}}(x^n_{\al}),$ for all $n\geq0.$

\begin{itemize}
\item If $\al\not=\al^{n}(\varepsilon), \al^{n}(1-\varepsilon)$ then clearly $L^{n+1}_{\al}=L_{\al}^n.$\\

\item If $\al=\al^{n}(\varepsilon)$ then $q^n_{\al^n(\varepsilon)}=q^{n+1}_{\al^n(\varepsilon)}$ and therefore

\begin{eqnarray*}
L^{n+1}_{\al}&=& \ln Df^{q^n_{\al}}(x^{n+1}_{\al})\\
& = &  \ln Df^{q^n_{\al}}(x^{n+1}_{\al}) + \ln Df^{q^n_{\al}}(x^{n}_{\al}) -\ln Df^{q^n_{\al}}(x^{n}_{\al})\\
& = & L^{n}_{\al} +\ln Df^{q^n_{\al}}(x^{n+1}_{\al}) - \ln Df^{q^n_{\al}}(x^{n}_{\al})\\
& = & L_{\al}^n + \O(\la^{\sqrt{n}}),
\end{eqnarray*}
by Theorem 3 of \cite{cunhasmania}.\\
\item If $\al=\al^{n}(1-\varepsilon)$ then $q^{n+1}_{\al^n(1-\varepsilon)}=q^{n}_{\al^n(1-\varepsilon)}+q^{n}_{\al^nn(\varepsilon)}.$ Note also that $f^{q^{n}_{\al(1-\varepsilon)}}(x^n_{\al^nn(1-\varepsilon)})\in I^n_{\al^n(\varepsilon)}.$ Therefore

\begin{eqnarray*}
L^{n+1}_{\al^n(1-\varepsilon)}&=& \ln Df^{q^{n+1}_{\al^n(1-\varepsilon)}}(x^{n+1}_{\al^n(1-\varepsilon)})\\
& = &  \ln Df^{q^n_{\al^n(\varepsilon)}}(f^{q^n_{\al^n(1-\varepsilon)}}(x^{n+1}_{\al^n(1-\varepsilon)})) + \ln Df^{q^n_{\al^n(1-\varepsilon)}}(x^{n+1}_{\al^n(1-\varepsilon)})\\
& = &  \ln Df^{q^n_{\al^n(\varepsilon)}}(f^{q^n_{\al^n(1-\varepsilon)}}(x^{n+1}_{\al^n(1-\varepsilon)})) - \ln Df^{q^n_{\al^n(\varepsilon)}}(x^{n}_{\al^n(\varepsilon)})+ L_{\al^n(\varepsilon)}^n\\
&  & + \ln Df^{q^n_{\al^n(1-\varepsilon)}}(x^{n+1}_{\al^n(1-\varepsilon)}) - \ln Df^{q^n_{\al^n(1-\varepsilon)}}(x^{n}_{\al^n(1-\varepsilon)})+L_{\al^n(1-\varepsilon)}^n \\
& = &  L_{\al^n(\varepsilon)}^n+ L_{\al^n(1-\varepsilon)}^n  + \O(\la^{\sqrt{n}}),
\end{eqnarray*}
by Theorem 3 of \cite{cunhasmania}.
\end{itemize}

This finishes the proof.

\cqd

Now we decompose the vector $L_n=(L_n)_{\al\in\A}$ as

$$L_n=L_n^s+ L_n^c + L_n^u\in E^s_n \oplus E^c_{n,\infty}\oplus E^u_n.$$

\begin{lemma}
\label{Lstable}
The sequence $\{L_n^s\}_{n\in\N}$ satisfies
$||L_n^s||=O(\lambda^{\sqrt{n}}).$ \end{lemma}

\dem
By Proposition \ref{hyp} we have for all $j,n\geq 0$ and for all $v\in E^s_j$ that
$$\|\Theta_{j+n-1}\cdot \Theta_{j+n-2}\cdots\Theta_{j}v\|\leq \frac{1}{C_2\cdot\mu^n}\|v\|.$$
Replacing this norm by the {\it adapted norm}, see \cite[Proposition 4.2]{shub}, that we still denote by $\|\cdot\|$ for simplicity, we can find $\mu>\tilde{\mu}>1$ such that for all $n\geq 0$ and for all $v\in E^s_n$ we have
$$\|\Theta_nv\|\leq\frac{1}{\tilde\mu}\|v\|.$$
By Proposition \ref{L} we have
$$\|L_n^s\|\leq \frac{1}{\tilde\mu}\|L_{n-1}^s\|+C\cdot \la^{\sqrt{n-1}}.$$
Applying this estimative n times we obtain
\begin{equation}\label{e1234} |L_n^s\|\leq \frac{1}{\tilde\mu^n}\cdot\|L_0^s\|+C\cdot \sum_{i=0}^{n-1}\frac{1}{\tilde\mu^i}\la^{\sqrt{n-i-1}}.\end{equation}
Note that
\begin{eqnarray}
\label{p0}
\frac{1}{\tilde\mu^n}\cdot\|L_0^s\|\to 0\; \text{ when } n\to\infty.
\end{eqnarray}
Now we analyse the last part of the \eqref{e1234}. Denote $$a_{n,i}=\frac{1}{\tilde\mu^i}\cdot \la^{\sqrt{n-i+1}}.$$ Then
\begin{eqnarray}
\label{razao}
\frac{a_{n,i+1}}{a_{n,i}}=\frac{1}{\tilde\mu}\cdot\la^{\sqrt{n-i-2}-\sqrt{n-i-1}}\leq \frac{1}{\tilde\mu}\cdot \la^{\frac{-1}{2\sqrt{n-i-2}}},
\end{eqnarray}
where the inequality above is given by Mean Value Theorem. Let $n_0\in\N$ be such that if $n-i-2\geq n_0$ then
$$\la^{\frac{-1}{2\sqrt{n-i-2}}}\leq \frac{1}{\tilde\mu}.$$
So

$$\sum_{i=0}^{n-1}\frac{1}{\tilde\mu^i}\la^{\sqrt{n-i-1}}=\sum_{i=0}^{n-n_0-2}\frac{1}{\tilde\mu^i}\la^{\sqrt{n-i-1}}+\sum_{i=n-n_0-1}^{n-1}\frac{1}{\tilde\mu^i}\la^{\sqrt{n-i-1}}.$$

The first sum is estimated by

\begin{eqnarray}
\label{p1}
\sum_{i=0}^{n-n_0-2}\frac{1}{\tilde\mu^i}\la^{\sqrt{n-i-1}}&=&\sum_{i=0}^{n-n_0-2} a_{n,i}\nonumber\\
&\leq& \sum_{i=0}^{n-n_0-2} \frac{1}{\tilde\mu^{2i}}\la^{\sqrt{n}} \nonumber\\
& \leq &\la^{\sqrt{n}}\cdot\sum_{i=0}^{n-n_0-2}\frac{1}{\tilde\mu^{2i}}, \text{ por \eqref{razao}}\nonumber\\
&\leq & \frac{1}{1-\tilde\mu^{-2}}\cdot \la^{\sqrt{n}}
\end{eqnarray}
and the second by
\begin{eqnarray}
\label{p2}
\sum_{i=n-n_0-1}^{n-1}\frac{1}{\tilde\mu^i}\la^{\sqrt{n-i-1}}&\leq&\sum_{i=n-n_0-1}^{n-1}\frac{1}{\tilde\mu^i}\nonumber\\
&\leq& \frac{1}{\tilde\mu^n}\sum_{i=n-n_0-1}^{n-1}\frac{1}{\tilde\mu^{i-n}}\nonumber\\
& \leq & C\cdot \frac{1}{\tilde\mu^n}.
\end{eqnarray}
Of \eqref{p0}, \eqref{p1} e \eqref{p2}, we get the result.

\cqd

\begin{lemma}
\label{Lunstable} There is $\lambda_5 \in [0,1)$ such that
the  sequence $\{L_n^u\}_{n\in\N}$ satisfies
$||L_n^u||=O(\lambda_5^{\sqrt{n}})$ \end{lemma}
\begin{proof}
The proof is similar to Lemma \ref{Lstable} and we use the adapted norm again. For all $n\geq0$ we have that
$$\|L_{n+1}^u\|\geq \tilde\mu\cdot\|L_n^u\|-C\cdot\la_5^{\sqrt{n}}.$$
Applying this estimative $k$ times, we obtain
$$\|L^u_{n+k}\|\geq \tilde\mu^k\|L^u_n\|-C\cdot \sum_{j=0}^{k-1}\tilde\mu^j\cdot\la_5^{\sqrt{n+k-1-j}},$$
and therefore
$$\|L_n^u\|\leq \frac{1}{\tilde\mu^k}\|L_{n+k}^u\|+C\cdot \sum_{j=0}^{k-1}\tilde\mu^{j-k}\cdot\la_5^{\sqrt{n+k-1-j}}.$$
Making $k=n,$ we have
$$\|L_n^u\|\leq \frac{1}{\tilde\mu^n}\|L_{2n}^u\|+C\cdot \sum_{j=0}^{n-1}\tilde\mu^{j-n}\cdot\la_5^{\sqrt{2n-1-j}}.$$
The sequence $\{L^u_{2n}\}_{n\in\N}$ is uniformly bounded, then

\begin{eqnarray}
\label{p0u}
\frac{1}{\tilde\mu^n}\|L_{2n}^u\|\to 0,\; \text{ when } n\to\infty.
\end{eqnarray}
Now note that
$$ \sum_{j=0}^{n-1}\tilde\mu^{j-n}\cdot\la_5^{\sqrt{2n-1-j}}=\sum_{s=0}^{n-1}\tilde\mu^{-1-s}\cdot\la_5^{\sqrt{n+s}}.$$
Let $a_s=\tilde\mu^{-1-s}\cdot\la_5^{\sqrt{n+s}}.$ Then

\begin{eqnarray}
\label{p1u}
\frac{a_{s+1}}{a_s}=\frac{1}{\tilde\mu}\cdot \la_5^{\sqrt{n+s+1}-\sqrt{n+s}}\leq \frac{1}{\tilde\mu}\cdot\la_5^{\frac{1}{2\sqrt{n+s+1}}}.
\end{eqnarray}
By  \eqref{p1u}

\begin{eqnarray}
\label{p2u}
&\sum_{s=0}^{n-1}&\tilde\mu^{-1-s}\cdot\la_5^{\sqrt{n+s}}\nonumber \\&=& \tilde\mu^{-1}\la_5^{\sqrt{n}}+\tilde\mu^{-2}\la_5^{\sqrt{n+1}}+\cdots + \tilde\mu^{-n}\la_5^{\sqrt{2n-1}}\nonumber \\
& \leq & \tilde\mu^{-1}\la_5^{\sqrt{n}}\left(1+\tilde\mu^{-1}\la_5^{\sqrt{n+1}-\sqrt{n}}+\cdots + \tilde\mu^{-n+1}\la_5^{\sqrt{2n-1}-\sqrt{n}}\right)\nonumber\\
&\leq &  \tilde\mu^{-1}\la_5^{\sqrt{n}}\left(1+\tilde\mu^{-1}\la_5^{\frac{1}{2\sqrt{n+1}}}+\ldots + \tilde\mu^{-n+1}\la_5^{\frac{1}{2\sqrt{n+1}}+\ldots + \frac{1}{2\sqrt{2n-1}}}\right)\;\nonumber\\
&\leq& \tilde\mu^{-1}\la_5^{\sqrt{n}}\left(1+\tilde\mu^{-1}+\ldots +\tilde\mu^{-n+1}\right)\nonumber\\
&\leq &  \tilde\mu^{-1}\la_5^{\sqrt{n}}\cdot \frac{1}{1-\tilde\mu^{-1}}.
\end{eqnarray}
Of \eqref{p0u} and \eqref{p2u} we get the result.
\end{proof}

Define for all $n\geq0$ the vector
$$\tilde{\omega}_{n}:=\Theta_0^{-1}\Theta_1^{-1}\cdots\Theta_{n-1}^{-1}(L_n^c)\in E^c_{0,\infty}.$$

By  Proposition \ref{quasi.iso} and Proposition \ref{L} it is  easy to see that $$\|\tilde{\omega}_{n+1}~-~\tilde{\omega}_n~\|=\O(\lambda^{\sqrt{n}}).$$ Therefore $\{\tilde{\omega}_{n}\}_{n\in\N}$ converges.

\begin{lemma}
\label{log.deriv}
Let $\omega=\lim_{n\to\infty}\tilde{\omega}_n\in E^c_{0,\infty}$ and $\omega^n\in E^c_{n,\infty}$ be the orbit given by Rauzy-Veech cocycle of $\omega,$ that is, $\omega^n=\Theta_0\cdots\Theta_{n-1}\omega.$ Then
$$\|\omega^n-L_n\|=\O(\la^{\sqrt{n}}).$$
\end{lemma}

\dem
By Lemma \ref{Lstable} and \ref{Lunstable} it is sufficient to estimate $\omega^n-L_n^c.$
\begin{eqnarray*}
\|\omega^n-L_n^c\|&=&\|\Theta_0\cdots\Theta_{n-1}\omega-L_n^c\|\\
&\leq & \|\left.\Theta_0\cdots\Theta_{n-1}\right|_{E^c_{0.\infty}}\|\cdot \|\omega-(\Theta_0\cdots\Theta_{n-1})^{-1}L_n^c\|\\
&\leq & C_{30}\cdot \left(1+\frac{1}{\ve_0}\right)\cdot \|\omega-\tilde{\omega}_n\|, \text{ by Proposition }\ref{quasi.iso}\\
&=& \O(\la^{\sqrt{n}}).
\end{eqnarray*}
\cqd

\subsection{Projective Metrics and Proof of Theorem \ref{teo3}}

The presentation follows Section 4.6 of \cite{viana}. Consider the convex cone $\mathbb{R}^{A}_+.$ Given any $\la,\gamma\in\mathbb{R}^{A}_+,$ define

$$a(\la,\gamma)=\inf_{\al\in\A}\frac{\la_\al}{\gamma_\al}\mbox{~~~and~~~}b(\la,\gamma)=\sup_{\beta\in\A}\frac{\la_\beta}{\gamma_\beta}.$$

The {\it projective metric} associated to $\mathbb{R}^{A}_+$ is defined by

$$\d_p(\la,\gamma)=\log\frac{b(\la,\gamma)}{a(\la,\gamma)}=\log\sup_{\al,\beta\in\A}\frac{\la_{\al}\gamma_{\beta}}{\gamma_{\alpha}\la_{\beta}}.$$

Follows easily from the definition that $\d_p$ satisfies, for all $\la,\gamma,\zeta\in\mathbb{R}^{A}_+$
\begin{itemize}
\item[(a)] $\d_p(\la,\gamma)=\d_p(\gamma,\al);$\\
\item[(b)] $\d_p(\la,\gamma)=\d_p(\al,\zeta)+\d_p(\zeta,\gamma);$\\
\item[(c)] $\d_p(\la,\gamma)\geq0;$\\
\item[(d)] $\d_p(\la,\gamma)=0$ if and only if there exists $t>0$ such that $\la=t\gamma.$
\end{itemize}

Let $G: \mathbb{R}^\A\to\mathbb{R}^\A$ be a linear operator such that $G(\mathbb{R}^{A}_+)\subset\mathbb{R}^{A}_+$ or, equivalently, $G_{\al\beta}\geq0$ for all $\al,\beta\in\A,$ where $G_{\al\beta}$ are the entries of the matrix $G.$ Then for all $\al,\gamma\in\mathbb{R}^{\A},$

$$\d_p(G(\la),G(\gamma))\leq \d_p(\la,\gamma).$$

Now, we define $g:\Delta_\A\to\Delta_\A$ by

$$g(\la)=\frac{G(\la)}{\sum_{\al\in\A}G(\la)_{\al}}=\frac{G(\la)}{\sum_{\al,\beta\in\A}G_{\al\beta}\la_{\beta}}.$$

We say $g$ is the projectivization of $G.$

The next proposition, whose proof can be found in \cite{viana}, ensures that if $g(\Delta_\A)$ has finite $\d_p-$diameter in $\Delta_\A$ then $g$ is a uniform contraction relative to the metric projective:

\begin{pro}
\label{contracao}
For any $\Delta>0$ there is $\kappa<1$ such that if the diameter of $G(\mathbb{R}^\A_+)$ relative to $\d_p$ is less than $\Delta$ then
$$\d_p(G(\la),G(\gamma))\leq\kappa\cdot\d_p(\la,\gamma)\quad\mbox{for all}\quad \la,\gamma\in\mathbb{R}^\A_+.$$
\end{pro}

Let $\zeta^n$ be  the partition vector of $f_n=R^n(f).$ Define the following linear operator  $T_n:\mathbb{R}^\A\to\mathbb{R}^\A$ whose the matrix is given by

$$
T_n=(T_n)_{ij}=\left\{
\begin{array}{ll}
1,& \mbox{~if~} i=j \mbox{~and~} i\not=\al^n(1-\varepsilon^nn)\\
\exp(\varepsilon^n\cdot L^n_{\al^n(1)}), & \mbox{~if~} i=j=\al^n(1-\varepsilon^n)\\
\exp((1-\varepsilon^n)\cdot L_{\al^n(1)}^n),& \mbox{~if~}  i=\al^n(\varepsilon^n), j=\al^n(1-\varepsilon^n)\\
0,& \mbox{otherwise,}
\end{array}
\right.
$$
where $L_{n,\al^n(1-\varepsilon)}$ is defined by \eqref{mean-deriv}.

\begin{lemma}
\label{sr}
Let $f\in\mathcal{B}_{k}^{2+\nu}$ with $\int_0^1D^2f(s)/ Df(s)ds=0.$ Then for all $n\geq0$
$$T_n\zeta^{n+1}=\zeta^n+\O(\la^{n}).$$
\end{lemma}

\dem
It follows easily by the definition of the Rauzy-Veech induction and by Theorem 3 of \cite{cunhasmania}.
\cqd

Note that $(T_n)_{ij}\geq0$ for all $n\geq0.$ Then we can define the  projectivization of $T_n,$ this is, the map $T_n^\nor: \Delta_{\A}\to\Delta_{\A}$ given by

\begin{eqnarray*}
T_n^\nor\zeta^{n+1}=\frac{T_n\zeta^{n+1}}{|T_n\zeta^{n+1}|_1},
\end{eqnarray*}
where $|\cdot|_1$ denote the sum-norm.

Let $\omega$ be  either as in Lemma \ref{log.deriv} (in this case $\omega \in E^c_{0,\infty}$) or a perturbation of it by a vector in $E^s_0$.    Define $\omega^n=\Theta_0\cdots\Theta_{n-1}\omega.$
Define $\tilde{T}_n$ and $\tilde{T}^\nor_n$ by

$$
\tilde{T}_n=(\tilde{T}_n)_{ij}=\left\{
\begin{array}{ll}
1,& \mbox{~if~} i=j \mbox{~and~} i\not=\al^n(1-\varepsilon^n)\\
\exp(\varepsilon^n\cdot \omega^n_{\al^n(1)}), & \mbox{~if~} i=j=\al^n(1-\varepsilon^n)\\
\exp((1-\varepsilon^n)\cdot\omega_{\al^n(1)}^n),& \mbox{~if~}  i=\al^n(\varepsilon^n), j=\al^n(1-\varepsilon^n)\\
0,& \mbox{otherwise,}
\end{array}
\right.
$$
and
$$\tilde{T}_n^\nor\zeta=\frac{\tilde{T}_n\zeta}{|\tilde{T}_n\zeta|_1}.$$

\begin{lemma}
\label{vecpart}
Let $f\in\mathcal{B}_{k}^{2+\nu}$ with $\int_0^1D^2f(s)/ Df(s)ds=0.$ Then for all $n\geq0$
$$\zeta^{n}=\tilde{T}^\nor_n\zeta^{n+1}+\O(\la^{\sqrt{n}}),$$
\end{lemma}

\dem
\begin{eqnarray*}
|\zeta^{n}-\tilde{T}^\nor_n(\zeta^{n+1})|\leq|\zeta^n-T^\nor_n\zeta^{n+1}|+|T^\nor_n\zeta^{n+1}-\tilde{T}^\nor_n\zeta^{n+1}|.
\end{eqnarray*}
Both terms above are of order $\la^{\sqrt{n}}$ by Lemma \ref{log.deriv}, Lemma \ref{sr} and Lemma 3.6, Lemma 3.7 of \cite{cunhasmania}.


\cqd

Given $n,m\in\N$ with $n<m$ define

$$\T_n^m:=\tilde{T}^{\nor}_n\cdots \tilde{T}^{\nor}_m.$$

By definition of  $\tilde{T}^{\nor}_n$ we know that
$$\T_n^m>0 \Leftrightarrow \Theta_{\pi^n,\ve^n}\cdots \Theta_{\pi^m,\ve^m}>0.$$

Then by \cite[Section 1.2.4]{coho} we have that for all $n\geq k(2d-3)$
\begin{eqnarray}
\label{contraction}
\T_{n-k(2d-3)}^{n-1}>0.
\end{eqnarray}
By \cite{cunhasmania} there exists $C$ such that $|\omega^n_\alpha|\leq C$ for every $n$ and $\alpha\in \A$.  This implies that  there exists $c_0$ such that for all $n$
$$\T_{n-k(2d-3)}^{n-1}\Delta_\A \subset K_{c_0},$$
where  $K_c\subset\Delta_\A$ is  the set of all $(\zeta_\al)_{\al \in \A}$ such that $\zeta_\al \geq c > 0$, for all $\al \in \A$. Note that $K_c$ is relatively compact in $\Delta_\A$ and by definition of $\d_p$
$$\mathrm{diam}(K_c)=\sup_{x,y\in K_c}\d_p(x,y)<\infty.$$
and therefore $\T_{n-k(2d-3)}^{n-1}$ is a contraction in the projective metric, and the rate of contraction can be taken uniformly for all $n$.  We will denote this rate  by $\kappa < 1$,

Let $c < c_0$ be such that   $\zeta^n_\al\geq c>0$ for all $n\geq 0$ and for all $\al\in\A.$ Such  $c$ does exist  by Lemma 3.6 and Lemma 3.7 of \cite{cunhasmania}. It is easy to see that the metrics $\d_p$ and $|\cdot|_1$ are equivalent in $K_c$.

Note that for every $n$ and $j$
$$\T_{j}^{n+1}\Delta_\A \subset \T_{j}^{n}\Delta_\A.$$
Moreover
$$diam \ \T_{j}^{n} \leq C \kappa^{(n-j)/k(2d-3)}.$$
In particular
$$\{ \tilde{\zeta}^j \} = \bigcap_{n\geq j} \T_{j}^{n}\Delta_\A$$
for some positive vector $\tilde{\zeta}^j \in K_{c_0}$. As a consequence
\begin{equation}\label{comb.equal} \tilde{T}_n^\nor\tilde{\zeta}^{n+1}=\tilde{\zeta}^{n}.\end{equation}

\begin{lemma}
\label{vp}
Let $f\in \mathcal{B}^{2+\nu}_{k}$ be such that  $\int_0^1D^2f(s)/ Df(s)ds=0$. Let  $f_A$ be  its affine model and let $\zeta^n,\tilde{\zeta}^n$ be  as above.  Then
$$|\zeta^n-\tilde{\zeta}^n|_1=\O(\la^{\sqrt{n/2}}).$$
\end{lemma}

\dem

First note that for every $n$
\begin{eqnarray*}
\d_p(\zeta^{n-k(2d-3)},\mathcal{T}_{n-k(2d-3)}^{n-1}\zeta^{n})&\leq&\d_p(\zeta^{n-k(2d-3)},\tilde{T}_{n-k(2d-3)}^\nor\zeta^{n-k(2d-3)+1})\\
& &+\d_p(\tilde{T}_{n-k(2d-3)}^\nor\zeta^{n-k(2d-3)+1},\mathcal{T}_{n-k(2d-3)}^{n-1}\zeta^{n})\\
& \leq & \O(\la^{\sqrt{n-k(2d-3)}})+\d_p(\zeta^{n-k(2d-3)+1},\mathcal{T}_{n-k(2d-3)+1}^{n-1}\zeta^{n}).
\end{eqnarray*}

Applying this  $k(2d-3)-$times, we obtain
\begin{eqnarray}
\label{0est}
\d_p(\zeta^{n-k(2d-3)},\mathcal{T}_{n-k(2d-3)}^{n-1}\zeta^{n})\leq\sum_{i=0}^{k(2d-3)-1} \O(\la^{\sqrt{n-k(2d-3)+i}}),
\end{eqnarray}
but

\begin{eqnarray}
\label{1est}
\sum_{i=0}^{k(2d-3)-1} \O(\la_5^{\sqrt{n-k(2d-3)+i}})&=& \O(\la^{\sqrt{n-k(2d-3)}})\sum_{i=0}^{k(2d-3)-1} \O(\la^{\sqrt{n-k(2d-3)+i}-\sqrt{n-k(2d-3)}})\nonumber\\
&\leq & k(2d-3)\cdot  \O(\la^{\sqrt{n-k(2d-3)}})\nonumber\\
&=& \O(\la^{\sqrt{n-k(2d-3)}}).
\end{eqnarray}

The Eq. \eqref{0est} jointly with Eq. \eqref{1est} give us
\begin{eqnarray}
\label{2est}
\d_p(\zeta^{n-k(2d-3)},\mathcal{T}_{n-k(2d-3)}^{n-1}\zeta^{n})= \O(\la^{\sqrt{n-k(2d-3)}}).
\end{eqnarray}

From  Eq. \eqref{2est} it is easy  to see that for all $1\leq j \leq \left[\frac{n}{k(2d-3)}\right]$

\begin{eqnarray}
&&d_p(\zeta^{n-jk(2d-3)},\mathcal{T}_{n-jk(2d-3)}^{n-1}\zeta^{n}) \nonumber \\
&\leq & \sum_{i=0}^{j-1}d_p(\mathcal{T}_{n-jk(2d-3)}^{n-ik(2d-3)-1}\zeta^{n-ik(2d-3)},\mathcal{T}_{n-jk(2d-3)}^{n-(i+1)k(2d-3)-1}\zeta^{n-(i+1)k(2d-3)}) \nonumber  \\
&\leq & \sum_{i=0}^{j-1}d_p(\mathcal{T}_{n-jk(2d-3)}^{n-(i+1)k(2d-3)-1}\mathcal{T}_{n-(i+1)k(2d-3)}^{n-ik(2d-3)-1}\zeta^{n-ik(2d-3)},\mathcal{T}_{n-jk(2d-3)}^{n-(i+1)k(2d-3)-1}\zeta^{n-(i+1)k(2d-3)}) \nonumber  \\
&=&\sum_{i=0}^{j-1}\kappa^{j-1-i}\O(\la^{\sqrt{n-(i+1)k(2d-3)}})\nonumber \\
& \leq & \O(\la^{\sqrt{n-jk(2d-3)}})\cdot  \sum_{i=0}^{j-1}\kappa^{j-1-i}\nonumber\\
& \leq & \frac{1}{1-\kappa}\cdot  \O(\la^{\sqrt{n-jk(2d-3)}})\nonumber.
\end{eqnarray}
so
\begin{equation}
\label{3est}
d_p(\zeta^{n-jk(2d-3)},\mathcal{T}_{n-jk(2d-3)}^{n-1}\zeta^{n})= \O(\la^{\sqrt{n-jk(2d-3)}}).
\end{equation}

Then
\begin{eqnarray}
\label{4est}
\d_p(\zeta^{n-jk(2d-3)},\tilde{\zeta}^{n-jk(2d-3)})& \leq & \d_p(\zeta^{n-jk(2d-3)}, \T_{n-jk(2d-3)}^{n-1}\zeta^n)\nonumber\\
& & + \d_p(\T_{n-jk(2d-3)}^{n-1}\zeta^n,\T_{n-jk(2d-3)}^{n-1}\tilde{\zeta}^n)\nonumber\\
&\leq & \O(\la^{\sqrt{n-jk(2d-3)}}) + \kappa^{j}\cdot \d_p(\zeta^n,\tilde{\zeta}^n)\nonumber\\
&\leq &  \O(\la^{\sqrt{n-jk(2d-3)}}) + \kappa^{j}\cdot \mathrm{diam}(K)
\end{eqnarray}

Taking $j=\left[\frac{n}{2k(2d-3)}\right]$ in the Eq. \eqref{4est} we have that
$$
\d_p(\zeta^{[n/2]},\tilde{\zeta}^{[n/2]})\leq  \O(\la^{\sqrt{[n/2]}})+\kappa^{\left[\frac{n}{2k(2d-3)}\right]}\cdot \mathrm{diam}(K),
$$
and therefore
\begin{eqnarray}
\label{5est}
\d_p(\zeta^{n},\tilde{\zeta}^{n})=\O(\la^{\sqrt{n}}).
\end{eqnarray}

\cqd

\begin{pro}\label{igualdade.maluca} We have
$$\sum_{i\in \mathcal{A}} e^{\omega_i}\tilde{\zeta}^0_i =1.$$
\end{pro}
\begin{proof} For $\zeta \in \Delta_{\mathcal{A}}$ and $n\in \mathbb{N}$  define
$$p_{n}(\zeta)= \frac{\sum_i e^{\omega^n_i}\zeta_i}{\sum_i \zeta_i}$$
Note that
$$p_n(\tilde{T}_n\zeta)= \frac{\sum_i e^{\omega^{n+1}_i}\zeta_i+e^{\omega^{n}_{\alpha^n(1)}}\zeta_{\alpha^n(1-\varepsilon^n)}}{\sum_i \zeta_i+e^{\omega^{n}_{\alpha^n(1)}}\zeta_{\alpha^n(1-\varepsilon^n)}}.$$
One can easily verify that
\begin{equation}\label{conv.maluca} |p_n(\tilde{T}_n\zeta) -1|= \frac{\sum_i \zeta_i}{\sum_i \zeta_i+ e^{\omega^{n}_{\alpha^n(1)}}\zeta_{\alpha^n(1-\varepsilon^n)}} |p_{n+1}(\zeta) -1|\end{equation}
By Lemma \ref{log.deriv} and \cite[Lemma 3.5]{cunhasmania} we have that $\sup_{n,\alpha} |\omega^n_\alpha| < \infty$. We have that  $\tilde{\zeta}^n \in K_{c}$  for every $n$. In particular
$$\sup_n \big|\frac{\sum_i \tilde{\zeta}_i^{n+1}}{\sum_i \tilde{\zeta}^{n+1}_i+ e^{\omega^{n}_{\alpha^n(1)}}\tilde{\zeta}^{n+1}_{\alpha^n(1-\varepsilon^n)}}\big|=\theta < 1.$$
By \cite{cunhasmania} we have that $R^nf$ is almost an affine g.i.e.m. so
$$\lim_n \frac{\sum_{i \in \mathcal{A}} e^{L^n_i} \zeta^n_i }{\sum_{i \in \mathcal{A}} \zeta^n_i} =1.$$
By Lemma \ref{log.deriv}  and Lemma \ref{vp} it follows that
$$\lim_n p_n(\tilde{\zeta}^n)=1.$$
By (\ref{conv.maluca})
$$|\sum_{i\in \mathcal{A}} e^{\omega_i}\tilde{\zeta}^0_i  -1|=|p_0(\tilde{\zeta}^0)-1|\leq \theta^n | p_n(\tilde{\zeta}^n)-1| \rightarrow_n 0.$$
\end{proof}

\begin{pro}{(See also \cite[Proposition 2.3]{mmy})}\label{affine} There is an unique  affine g.i.e.m. $f_A$ with domain $[0,1]$, whose combinatorics is $\{\pi^i,\ve^i\}_{i\in\N}$ and the slope vector is $\omega$. \end{pro}
\begin{proof}  It follows from Proposition \ref{igualdade.maluca} that the  slope vector $\omega$ and partition vector $\tilde{\zeta}^0$ defines an affine g.i.e.m.  By (\ref{comb.equal})
such affine g.i.e.m. has  combinatorics $\{\pi^i,\ve^i\}_{i\in\N}$.  To show the uniqueness, note that if $g\colon [0,1]\rightarrow [0,1]$ is an affine g.i.e.m. with slope $\omega$ and combinatorics $\{\pi^i,\ve^i\}_{i\in\N}$ then  the partition vectors $\hat{\zeta}^n$ of $R^ng$ satisfies $ \tilde{T}_n^\nor\hat{\zeta}^{n+1}=\hat{\zeta}^{n}$. In particular
$$\hat{\zeta}^0 \in   \bigcap_{n\geq 0} \T_{0}^{n}\Delta_\A =\{ \tilde{\zeta}^0 \}. $$\end{proof}


\begin{rem}  For each $\omega$ that is a sum of the vector given by Lemma \ref{log.deriv} and a vector in $E^s_0$ we  constructed   the unique  affine interval exchange map $f_A$ given by Proposition \ref{affine}. Each one of these affine interval exchange maps  is   called a {\it weak affine model} of $f$. So the weak affine model is not unique.  \end{rem}

From now on we assume without loss of generality that $f$ has only one descontinuity that will be denoted by $\partial I_{\al^*}.$

\begin{lemma}
\label{afim.break}
Suppose that $f\in\mathcal{B}^{2+\nu}_{k}$ satisfies $\int_0^1D^2(f)(s)/ Df(s)ds=0$ and let $f_A$ be a weak  affine model of $f$. Then $BP_f(\partial I_\alpha)=BP_{f_A}(\partial \tilde{I}_\al)$ for all $\al\in\A$ such that $\al\not=\al^*$ and $\pi_0(\al)>1.$
\end{lemma}

\dem
Denote by $n_s$ the sequence of times of renormalization such that $R^s_{\rot}f=R^{n_s}f$ and note that for all $s>0,$ $\partial I^{n_s}_{\al^*}$ is the unique point of descontiuity of $f.$ Let $\gamma\in\mathcal{A}$ such that $\gamma\not=\al^*$ and $\pi_0^{n_s}(\gamma)>1.$ As $f$ has no connection we have that there is a unique $1\leq j_{\gamma}<q^{n_s}_\gamma$ and unique $\alpha\in\A$ such that $f^{j_\gamma}(\partial I^{n_s}_\gamma)=\partial I_\al.$

We claim that $\al\not=\al^*$ and $\pi_0(\al)>1.$ In fact, if $f^{j_\gamma}(\partial I^{n_s}_{\gamma})=\partial I_{\al^*}$ then $f^{j_\gamma+1}(\partial I^{n_s}_{\gamma})=0=f^{q^{n_s}_{\al^*}}(\partial I^{n_s}_{\al^*})$ which is absurd because $f$ has no connection. If $\pi_0(\al)=1$ then $f^{j_\gamma}(\partial I^{n_s}_\gamma)=\partial I_\al=f(\partial I_{\al^*})$ which is absurd by the same reason as in  the previous case.

By definition of $n_s$ we have that for all $\beta\in\A$ such that $\pi_0^{n_s}(\beta)+1=\pi_0^{n_s}(\gamma)$ then $q^{n_s}_{\beta}=q^{n_s}_{\gamma}.$ Therefore
\begin{eqnarray}
\label{ig.break}
BP_{R^{n_s}f}(\partial I^{n_s}_{\gamma})&=&\ln D_-R^{n_s}f(\partial I^{n_s}_{\gamma})-\ln D_+R^{n_s}f(\partial I^{n_s}_{\gamma})\nonumber\\
& = & \sum_{i=0}^{q^{n_s}_\beta-1}\ln D_-f(f^i(\partial I^{n_s}_{\gamma}))- \sum_{i=0}^{q^{n_s}_\gamma-1}\ln D_+f(f^i(\partial I^{n_s}_{\gamma}))\nonumber\\
&= & \sum_{i=0}^{q^{n_s}_\gamma-1}\Big(\ln D_-f(f^i(\partial I^{n_s}_{\gamma}))-\ln D_+f(f^i(\partial I^{n_s}_{\gamma}))\Big)\nonumber\\
&=&\ln D_-f(f^{j_\gamma}(\partial I^{n_s}_{\gamma}))-\ln D_+f(f^{j_\gamma}(\partial I^{n_s}_{\gamma}))\nonumber\\
& =& \ln D_-f(\partial I_\al)-\ln D_+f(\partial I_\al)\nonumber\\
& = & BP_f(\partial I_\al)
\end{eqnarray}

As $f$ and its affine model has the same combinatorics we have
$$BP_{R^{n_s}f_A}(\partial \tilde{I}^{n_s}_{\gamma})= BP_{f_A}(\partial \tilde{I}_\al).$$
Then
\begin{eqnarray*}
BP_f(\partial I_\al)- BP_{f_A}(\partial \tilde{I}_\al)& = & BP_{R^{n_s}f}(\partial I^{n_s}_{\gamma})-BP_{R^{n_s}f_A}(\partial \tilde{I}^{n_s}_{\gamma})\\
&=& \ln D_-R^{n_s}f(\partial I^{n_s}_{\gamma})-\ln D_+R^{n_s}f(\partial I^{n_s}_{\gamma})\\
& & -\ln D_-R^{n_s}f_A(\partial \tilde{I}^{n_s}_{\gamma})+\ln D_+R^{n_s}f_A(\partial \tilde{I}^{n_s}_{\gamma})\\
& =& L_{n_s,\beta}-L_{n_s,\gamma}+\mathrm{O}(\la^{\sqrt{n_s}})-\omega^{n_s}_\beta+\omega^{n_s}_\gamma\\
&=& \mathrm{O}(\la^{\sqrt{n_s}}),\; \text{ by Proposition \ref{log.deriv}}.
\end{eqnarray*}

\cqd

The Lemma \ref{afim.break}  gives us that $f$ and $f_A$ have $d-2$ identical  breaks.
\begin{proof}[Proof of Theorem \ref{teo2}]
For simplicity we assume $\A=\{1,2,...,d\}$ and denote by $j_0\in\A$ the letter such that $\partial I_{j_0}$ is the descontinuity of $f_A.$ As $f$ and $g$ are break-equivalents, by the Lemma \ref{afim.break} we have
$$\omega_{i+1}^f-\omega_i^f=\omega_{i+1}^g-\omega_i^g\; \text{ for every }\; i\in\A\;\text{ such that }\;i\not=j_0-1,d,$$
which is equivalent to
$$\omega_1^f-\omega_1^g=...= \omega_{j_0-1}^f-\omega_{j_0-1}^g,$$
$$\omega_{j_0}^f-\omega_{j_0}^g=...= \omega_{d}^f-\omega_{d}^g,$$
where we choose  $\omega^f=(\omega^f_i)_{i\in\A}\in E^c_{0,\infty}$ and $\omega^g=(\omega^g_i)_{i\in\A}\in E^c_{0,\infty}$ as  the slope-vectors of the weak affine models $f_A$ e $g_A$ respectively.

Denoting by $v:=\omega_1^f-\omega_1^g$ and by $\tilde{v}:=\omega_{j_0}^f-\omega_{j_0}^g$ we have that
$$\omega^f-\omega^g=(\omega_i^f-\omega_i^g)_{i\in\A}=(v,...,v,\underbrace{\tilde{v}}_{j_0-\text{position}},...,\tilde{v})\in E^c_{0,\infty},$$
that is, the vector $\omega^f-\omega^g$ can be viewed as the slope-vector of a affine interval exchange maps with two intervals. So we have $(v,\tilde{v})\in E^c_{0,\infty}(2),$ where $E^c_{0,\infty}(2)$ is the central space defined by the renormalization of two intervals.  As $\dim E^c_{0,\infty}(2)=0$ we have $v=\tilde{v}=0$ and then $\omega^f=\omega^g.$ By Proposition \ref{affine} we have that $f_A=g_A$.
\end{proof}

The  next result estimate the distance between the image partition vectors of  $R^nf$ and  $R^nf_A.$

\begin{lemma}
\label{vpi}
Suppose that $f\in\mathcal{B}^{2+\nu}_{k}$ satisfies $\int_0^1D^2(f)(s)/ Df(s)ds=0$ and let $f_A$ be  a weak affine model of $f$. Then
$$\left||R^nf(I^n_\al)|-|R^nf_A(\tilde{I}^n_\al)|\right|=\O(\la^{\sqrt{n}})\; \text{ for all } \al\in\A.$$
\end{lemma}

\dem
It follows  from  Lemma \ref{log.deriv} and  Lemma \ref{vp}.
\cqd

Now note that by Theorem 3 of \cite{cunhasmania} we have
\begin{eqnarray}
\label{first.parcel}
\|Z_{I^n_\al}R^nf-Z_{\tilde{I}^n_\alpha}R^nf_A\|_{C^2}=\O(\la^{\sqrt{n}}), \text{ for all } n\geq0.
\end{eqnarray}

Theorem \ref{teo1} follows from \eqref{first.parcel}, Lemma \ref{vp} and Lemma \ref{vpi}.

\begin{proof}[Proof of Theorem \ref{teo3}.]
Let $f$ and $g$ as in the assumptions  of Theorem \ref{teo3}. Then by Theorem \ref{teo2} we have that $f_A=g_A.$ Therefore

$$d_{C^2}(R^nf,R^ng)\leq d_{C^2}(R^nf,R^nf_A)+d_{C^2}(R^ng_A,R^ng)=\O(\lambda^{\sqrt{n}}).$$

\end{proof}

\section{ Smoothness of the conjugacy}

To simplify the statements, denote by $\mathcal{B}^{2+\nu}_{k,*}$ the set of all $f \in \mathcal{B}^{2+\nu}_{k}$ such that $\int_0^1D^2(f)(s)/ Df(s)ds=0$.

Let $f,g\in\mathcal{B}^{2+\nu}_{k,*}$ be a g.i.e.m. with the  the same combinatorics. Then there is a orientation preserving homeomophism $h:[0,1)\to [0,1)$ that conjugates $f$ and $g$, that is,
\begin{eqnarray}
\label{conj}
g\circ h= h\circ f,
\end{eqnarray}
such that $h$ maps  break points of $f$ into  break points of $g.$

\subsection{Cohomological equation} The {\it cohomological equation}  associated to \eqref{conj} is
\begin{eqnarray}
\label{coh}
\ln Dg\circ h-\ln Df=\psi\circ f -\psi,
\end{eqnarray}
where $\psi: [0,1]\to \mathbb{R}$ is called the solution of \eqref{coh} if it exists.

For all $x\in[0,1)$ define $i_n(x):=\min\{i\geq0; f^i(x)\in I^n\}.$ Let
\begin{eqnarray*}
\psi_n(x)&:=&\sum_{i=0}^{i_n(x)-1}\ln Df(f^i(x))-\ln Dg (h\circ f^i(x)).
\end{eqnarray*}

\begin{lemma}
\label{unifor} Let $f,g\in\mathcal{B}^{2+\nu}_{k,*}$ be  g.i.e.m. with the  same combinatorics  and they admit the same weak affine model (they are not necessarily  break-equivalents). Then the  sequence $\psi_n$ is uniformly convergent. Indeed
\begin{equation}\label{estmaluca} |\psi_{n+1}(x)-\psi_n(x)|=\O(\la^{\sqrt{n}}).\end{equation}
\end{lemma}

\dem
We will show that $\psi_n$ is Cauchy. Suppose that $i_n(x)<i_{n+1}(x).$ Then $i_{n+1}(x)=i_n(x)+q^n_{(\pi_0^n)^{-1}(d)}.$ Therefore

$$
\begin{array}{l}
\psi_{n+1}(x)-\psi_n(x) =  \displaystyle\sum_{i=i_{n}(x)}^{i_{n+1}(x)-1}\ln Df(f^{i}(x))-\ln Dg(h\circ f^{i}(x))\vspace{0.3cm}\\
=  \ln Df^{i_{n+1}(x)-i_{n}(x)}(f^{i_{n}(x)}(x))-\ln Dg^{i_{n+1}(x)-i_{n}(x)}(h\circ f^{i_{n}(x)}(x))\vspace{0.3cm}\\
=   \ln Df^{q^n_{(\pi_0^n)^{-1}(d)}}(f^{i_{n}(x)}(x))-\ln Dg^{q^n_{(\pi_0^n)^{-1}(d)}}(h\circ f^{i_{n}(x)}(x))\vspace{0.3cm}\\
=\ln DR^nf(f^{i_{n}(x)}(x))-\ln DR^ng(h\circ f^{i_{n}(x)}(x))\vspace{0.3cm}\\
=\O(\la^{\sqrt{n}}),\; \text{ by Proposition }\ref{log.deriv}.
\end{array}
$$

\cqd

Let $\psi(x)=\lim_{n \to \infty}\psi_n(x).$

\begin{lemma}\label{sol.coh} Let $f,g\in\mathcal{B}^{2+\nu}_{k,*}$ be  g.i.e.m. with the  same combinatorics. Assume that  they are   break-equivalents. Then $\psi:[0,1)\to\mathbb{R}$ is continuous and it is the solution of \eqref{coh}.
\end{lemma}
\dem
It is easy to check that $\psi$ is solution of the \eqref{coh}.
Note that $i_n:[0,1)\to\mathbb{N}$ is continuous in the  interior of each element of the partition $\mathcal{P}^n.$ As a consequence $\psi_n$ is continuous in the interior of each element of the partition. Let $x \in [0,1]$. There are four cases. \\ \\
\noindent {\it Case i.} Suppose that $f^n(x)\not\in\cup_{\al\in\A}\partial I_\al$ for all $n$. Then $\psi_n$ is continuous at $x$ for all $n$ and by Lemma \ref{unifor} $\psi$ is continuous at $x.$ Moreover for large $n$ we have that $i_n(f(x))=i_n(x)-1$, so it is easy to see that $$\ln Dg\circ h(x)-\ln Df(x)=\psi_n\circ f(x) -\psi_n(x).$$
Taking the limit on $n$ we obtain \eqref{coh} for $x$ in {\it Case i.}  \\ \\
{\it Case ii.} Suppose that $x=0$. Then $\psi_n(0)=0$, so $\psi(0)=0$.  Let $y > 0$, with $y \in I^n$. Then $i_j(y)=0$ for every $j\leq n$, so  $\psi_j(y)=0$ for every $j\leq n$.  In particular
$$\psi_{n+p}(y)= \sum_{j=0}^{p} \psi_{n+j+1}(y)-\psi_{n+j}(y),$$
so by (\ref{estmaluca})
$$|\psi_{n+p}(y)|\leq C\sum_{j=0}^\infty \lambda^{\sqrt{n+j}} \rightarrow_n 0.$$
Consequently
$$\lim_n \sup_{y \in I^n} |\psi(y)| =0,$$
so $\psi$ is continuous at $x=0$.  \\ \\

\noindent {\it Case iii.} Suppose that there is $k_0\geq 0$ such that $f^{k_0+1}(x)=0$.  Note that $f(0)$ falls in Case i., so $\psi$ is continuous in it. Since there are not wandering intervals,  the points $x$ in {\it Case i.} are dense in $I$. Let $x_n$ be a sequence of points in {\it Case i.}  such that  $\lim_n x_n=0$.  Recall that $x_n$ satisfies
$$\ln Dg\circ h(x_n)-\ln Df(x_n)=\psi\circ f(x_n) -\psi(x_n).$$
Using Cases {\it i.} and {\it ii.}, we can take  the limit on  $n$  to obtain
\begin{equation}\label{igualdade} \ln Dg(0_+)-\ln Df(0_+)=\psi\circ f(0) -\psi(0)= \psi(f(0)).\end{equation}
Let $$n_0= \min \{ n\geq 0 \ s.t. \  f^i(x) \not\in I^n, \ for \ every \ i\leq k_0  \}$$
Then  $i_n(x_+)=k_0+1$ for every $n\geq n_0$. By (\ref{estmaluca}) it follows that
 $$\psi(x_+)=  \sum_{i=0}^{k_0}\ln Df(f^i(x))-\ln Dg (h\circ f^i(x)). $$
On the other hand, $i_n(x_-)=i_ {n}(f(0))+k_0+2$ for $n\geq n_0$.
 So by (\ref{igualdade})
 \begin{eqnarray*}
 \psi(x_-)&=& \sum_{i=0}^{k_0}\ln Df(f^i(x))-\ln Dg (h\circ f^i(x)) +  \ln Df(1_-)-\ln Dg(1_-) \\
 &+& \lim_n  \sum_{i=k_0+2}^{i_n(x_-)-1}\ln Df(f^i(x_-))-\ln Dg (h\circ f^i(x_-))\\
 &=&  \sum_{i=0}^{k_0}\ln Df(f^i(x))-\ln Dg (h\circ f^i(x)) +  \ln Df(1_-)-\ln Dg(1_-) \\
 &+& \lim_n \sum_{i=0}^{i_n(f(0))-1}\ln Df(f^i(f(0)))-\ln Dg (h\circ f^i(f(0))) \\
 &=& \sum_{i=0}^{k_0}\ln Df(f^i(x))-\ln Dg (h\circ f^i(x)) +  \ln Df(0_+)-\ln Dg(0_+) \\
 &+& \psi(f(0))\\
 &=& \sum_{i=0}^{k_0}\ln Df(f^i(x))-\ln Dg (h\circ f^i(x)) +  \ln Df(0_+)-\ln Dg(0_+) \\
 &+& \psi(f(0))=  \psi(x_+). \end{eqnarray*}

 We conclude that $\psi$ is continuous at $x$. \\

\noindent   {\it Case iv.} Now suppose that there is $k_0$ such that $f^{k_0}(x)=\partial I_\beta$ for some $\beta\in\A$, but $f^k(x)\not= 0$ for every $k$.  In particular $f$ is continuous at $f^k(x)$, for every $k$.  Then $i_n(x_+)=i_n(x_-)$ for every $n$ and
\begin{eqnarray*}
\psi(x_+)-\psi(x_-)&=&\lim_{n\to \infty} \psi_n(x_+)-\lim_{n\to \infty} \psi_n(x_-)\\
& = &\ln Df (f^{k_0}(x))- \ln Dg (h\circ f^{k_0}(x)) \\
& & -\ln Df ( f^{k_0}(x))+\ln Dg(h\circ f^{k_0}(x))\\
& = & \ln\frac{Df (f^{k_0}(x))}{Df(f^{k_0}(x))}-\ln\frac{Dg (h\circ f^{k_0}(x))}{Dg (h\circ f^{k_0}(x))}\\
& = & 0.
\end{eqnarray*}
\cqd

\subsection{Conjugacies} The next step is to show that the conjugacy $h$ is Lipschitz if $f$ and $g$ have the same weak affine model.
\begin{lemma}
\label{const} Let $f,g\in\mathcal{B}^{2+\nu}_{k,*}$ be  g.i.e.m. with   the same combinatorics  and they admit the same weak affine model (they are not necessarily  break-equivalents). Then there is $C_8>0$ such that for all $\beta\in\A$
$$\lim_{n\to\infty}\frac{|R^ng(h(I^n_{\beta}))|}{|R^nf(I^n_{\beta})|}=C_8,$$
and this convergence is uniform on $\beta$.
\end{lemma}

\dem Let  $f_A$ be a weak  affine model of  $f$ and $g$. Let $\hat{h}$ be the  corresponding conjugacy, that is,  $f_A\circ \hat{h} = \hat{h} \circ f$. By the Mean Value Theorem
\begin{eqnarray}
\label{pri}
\ln \frac{|R^nf_A(\hat{h}(I^n_\beta))|}{|R^nf(I^n_\beta)|}&=&\ln\frac{DR^nf_A(\hat{h}(y))}{DR^nf(y)}+\ln\frac{|\hat{h}(I^n_\beta)|}{|I^n_\beta|}\nonumber\\
&=& \O(\la^{\sqrt{n}})+\ln\frac{|\hat{h}(I^n_\beta)|}{|I^n_\beta|}
\end{eqnarray}
Now we show that the second term converges. By Lemma \ref{vp} we have
\begin{eqnarray}
\label{co}
\frac{|I^n_\beta|}{|I^n|}=\frac{|\hat{h}(I^n_\beta)|}{|\hat{h}(I^n)|}(1+\O(\la^{\sqrt{n}}))
\end{eqnarray}
and
$$\frac{|R^nf(I^n_\beta)|}{|I^n|}=\frac{|R^nf_A(\hat{h}(I^n_\beta))|}{|h(I^n)|}(1+\O(\la^{\sqrt{n}}).$$
Now suppose that $R^nf$ is type 0. Then
\begin{eqnarray}
\label{comp}
\frac{|I^{n+1}|}{|I^n|}&=&\frac{|I^n|-|R^nf(I^n_{\al^n(1)})|}{|I^n|}\nonumber\\
&=&1-\frac{|R^nf(I^n_{\al^n(1)})|}{|I^n|}\nonumber\\
& = & 1-\frac{|R^nf_A(\hat{h}(I^n_{\al^n(1)}))|}{|\hat{h}(I^n)|}(1+\O(\la^{\sqrt{n}}))\nonumber\\
& = & \frac{|\hat{h}(I^{n+1})|}{|\hat{h}(I^n)|}+\O(\la^{\sqrt{n}}).
\end{eqnarray}
The case in which $R^nf$ is type 1 is  analogous. From \eqref{comp} we obtain

\begin{eqnarray}
\label{comp1}
\frac{|\hat{h}(I^{n+1})|}{|\hat{h}(I^n)|}=\left(1+\O(\la^{\sqrt{n}})\right)\cdot \frac{|I^{n+1}|}{|I^n|}.
\end{eqnarray}
Therefore

\begin{eqnarray*}
\ln\frac{|\hat{h}(I^{n})|}{|I^n|}&=&\ln\frac{|\hat{h}(I^{0})|}{|I^0|}+\sum_{i=1}^n\ln\left[\frac{|\hat{h}(I^{i})|}{|\hat{h}(I^{i-1})|}\cdot\frac{|I^{i-1}|}{|I^i|}\right]\\
&=&\ln\frac{|\hat{h}(I^{0})|}{|I^0|}+\sum_{i=1}^n\ln[1+\O(\la^{\sqrt{n}})].
\end{eqnarray*}
The sum above is summable, thus
\begin{eqnarray}
\label{lim}
\lim_{n\to\infty}\ln\frac{|\hat{h}(I^{n})|}{|I^n|}=\ln C_8.
\end{eqnarray}
Therefore
\begin{eqnarray*}
\ln\frac{|\hat{h}(I^n_\beta)|}{|I^n_\beta|}&=& \ln\frac{|\hat{h}(I^n_\beta)|}{|\hat{h}(I^n)|}+\ln\frac{|\hat{h}(I^n)|}{|I^n|}+\ln\frac{|I^n|}{|I^n_\beta|}\\
&=&\ln\left(1+\O(\la^{\sqrt{n}})\right)+\ln\frac{|\hat{h}(I^n)|}{|I^n|}.
\end{eqnarray*}
Then there is $C_f >0$ such that
\begin{eqnarray}
\lim_{n\to\infty}\ln\frac{|\hat{h}(I^n_\beta)|}{|I^n_\beta|}=\ln C_f.
\end{eqnarray}
So  by (\ref{pri})
\begin{equation}\label{abcd} \lim_n \frac{|R^nf_A(\hat{h}(I^n_\beta))|}{|R^nf(I^n_\beta)|}= C_f.\end{equation}
Since  $f_A$ is also  an weak affine model of $g$, so there exists a conjugacy $\tilde{h}$ between $f_A$ and $g$,  $\tilde{h}\circ f_A = g\circ \tilde{h}$. As in the proof of (\ref{abcd})   we can show that there is $C_g > 0$ such that

 $$\lim_n \frac{|R^ng(\tilde{h}\circ \hat{h}(I^n_\beta))|}{|R^nf_A(\hat{h}(I^n_\beta))|}= C_g.$$
Since $h= \tilde{h}\circ \hat{h}$ it follows that

$$\lim_n \frac{|R^ng(h(I^n_\beta))|}{|R^nf(I^n_\beta)|}= C_f\cdot C_g=C_8.$$
\cqd
\begin{lemma}
\label{lip} Let $f,g\in\mathcal{B}^{2+\nu}_{k,*}$ be  g.i.e.m. with   the same combinatorics  and they admit the same weak affine model (they are not necessarily  break-equivalents). The conjugacy $h:[0,1)\to[0,1)$ is Lipschitz.
\end{lemma}

\dem Let $J_n\in \mathcal{P}^n.$ Note that $i_n$ is constant on $J_n.$  By the Mean Value Theorem there exist  $y, y' \in J_n$ such that
\begin{eqnarray*}
\frac{|h(J_n)|}{|J_n|}&=&\frac{Df^{i_n}(y')}{Dg^{i_n}(h(y))}\cdot  \frac{|g^{i_n}(h(J_n))|}{|f^{i_n}(J_n)|}.\\
\end{eqnarray*}
As $\{f^j(J_n)\}_{j=0}^{i_n-1}$ is a pairwise disjoint family of intervals it follows that

$$\exp(-V)\leq\frac{Df^{i_n}(y')}{Df^{i_n}(y)}\leq \exp(V),$$
where $V=\mathrm{Var}(\log Df).$ So
$$\frac{Df^{i_n}(y')}{Dg^{i_n}(h(y))}\leq\exp(V)\frac{Df^{i_n}(y)}{Dg^{i_n}(h(y))}\cdot\exp(\psi_n(y)).$$
Since that $f^{i_n}(J_n)=R^nf(I^n_\alpha)$ and $g^{i_n}(h(J_n))=R^ng(h(I^n_\alpha))$ for some $\al\in\A$ we have
$$C_9:=\sup_{n}\sup_{J_n\in\mathcal{P}_n}\frac{|g^{i_n}(h(J_n))|}{|f^{i_n}(J_n)|}=\sup_n\sup_{\al\in\A}\frac{|R^ng(h(I^n_\alpha))|}{|R^nf(I^n_\alpha)|}$$
is finite by Lemma \ref{const}.

Note that $C_{10}:=\sup_n\sup_{y\in[0,1)}\{\exp(\psi_n(y))\}$ is also finite.  Then
\begin{eqnarray*}
\frac{|h(J_n)|}{|J_n|}&\leq&C_9\cdot\exp(V)\frac{Df^{i_n}(y)}{Dg^{i_n}(h(y))}\\
&\leq & C_9\cdot\exp(V)\cdot\exp(\psi_n(y))\\
&\leq&C_9\cdot C_{10}\cdot\exp(V)
\end{eqnarray*}
Therefore there is $C_{11}>0$ such that for $n$ large enough
$$|h(J_n)|\leq C_{11}\cdot |J_n|.$$
Let $x,y\in[0,1)$ be such that $x<y$ and $A=[x,y).$ Define$$\mathcal{F}_1=\{J_1\in\mathcal{P}^1\; \setminus\; J_1^s\subset A\}=\{J_1^1,\ldots,J_1^{s_1}\}$$
and
\begin{eqnarray*}
\mathcal{F}_n&=&\{J_n\in\mathcal{P}^n\; \setminus\; J_n\subset A \text { and } J_n\cap J =\varnothing  \text{ for every } J \in \mathcal{F}_{i}, \ i < n\}\\
&=&\{J_n^1,\ldots,J_n^{s_n}\}.
\end{eqnarray*}
Note that if $n\not=m$ then $J_n^s\cap J_m^r=\varnothing$ for all $1\leq s\leq s_n, \; 1\leq r\leq s_m.$

It is clear that $$\bigcup_{i=1}^n\bigcup_{j=1}^{s_i}J^j_i\nearrow A \;\text{   and   }\; \bigcup_{i=1}^n\bigcup_{j=1}^{s_i}h(J^j_i)\nearrow h(A).$$Therefore
\begin{eqnarray*}
|h(A)| & = & \lim_n \sum_{i=1}^n\sum_{j=1}^{s_i}|h(J^j_i)|\\
&\leq & C_{10}\cdot \lim_n \sum_{i=1}^n\sum_{j=1}^{s_i}|J^j_i|\\
&\leq & C_{10}\cdot |A|.
\end{eqnarray*}

\cqd

\begin{lemma} \label{dist45} Let $f,g\in\mathcal{B}^{2+\nu}_{k,*}$ be  g.i.e.m. with   the same combinatorics and assume they admit  the same weak affine model (they are not necessarily  break-equivalents). Let $x$ in $[0,1)$, and let $J_n \in \mathcal{P}^n$ be  such that $x \in J_n$. For every $y_n,  y_n' \in J_n$  we have
$$\lim_n  \frac{Df^{i_n(x)}(y_n')}{Df^{i_n(x)}(y_n)}=1.$$
\end{lemma}
\begin{proof} Note that $f^{i_j}(J_n)\in \mathcal{P}^n$ and $f^{i_j}(J_n) \subset I^j_\beta$, for some $\beta \in \mathcal{A}$. By \cite{cunhasmania}, there exists $\theta < 1$ such that
$$\frac{|f^{i_j}y_n-f^{i_j}y_n'|}{|I^j_\beta|} \leq  \theta^{n-j}. $$
By \cite{cunhasmania} there exists $C$ such that for every $z,z' \in I^n_\beta$ we have
$$\big| \ln \frac{D(R^jf)(z')}{D(R^jf)(z)}\big| \leq C \frac{|z'-z|}{|I^n_\beta|}.$$
Again by \cite{cunhasmania} we have
$$\big| \ln \frac{D(R^jf)(z')}{D(R^jf)(z)}\big|=O(\lambda^{\sqrt{j}}).$$
So
 \begin{eqnarray*}
 \ln \frac{Df^{i_n(x)}(y_n')}{Df^{i_n(x)}(y_n)}&=&\sum_{j\leq n/2, \ i_{j-1} < i_j} \ln \frac{D(R^jf)(f^{i_j}(y'_n))}{D(R^jf)(f^{i_j}(y_n))} \\
 &+& \sum_{n/2< j\leq n, \ i_{j-1} < i_j} \ln \frac{D(R^jf)(f^{i_j}(y'_n))}{D(R^jf)(f^{i_j}(y_n))} \\
 &=&    O(\sum_{ j\leq n/2}  \theta^{n-j}) +  O( \sum_{n/2< j\leq n}  \lambda^{\sqrt{j}})\\
 &\rightarrow_n& 0
\end{eqnarray*}
\end{proof}

\begin{pro}\label{c_1} Let $f,g\in\mathcal{B}^{2+\nu}_{k,*}$ be  g.i.e.m. with the  same combinatorics. Assume that  they are   break-equivalents.  Then the conjugacy $h:[0,1)\to[0,1)$ is $C^1.$
\end{pro}
\dem
 Note that $Dh(x)$ exists for almost every $x\in[0,1)$ and by Lemma \ref{lip} the map $h$ is the integral of its derivative. Let $x_0\in[0,1)$ be such that $Dh(x_0)$ exists. For all $n\geq0$ there is $J_n\in\mathcal{P}^n$ such that $x_0\in J_n.$ Then
\begin{eqnarray}
\label{deriv}
\lim_{n\to\infty}\frac{|h(J_n)|}{|J_n|}=Dh(x_0).
\end{eqnarray}
Let $\alpha\in\A$ be such that
$$f^{i_n(x_0)}(J_n)=R^nf(I^n_{\al}).$$
We also have that $g^{i_n(x_0)}(h(J_n))=R^ng(h(I^n_{\al})).$ Let $y'\in J_n$ be  such that
$$|J_n|=\frac{|R^nf(I^n_{\al})|}{Df^{i_n(x_0)}(y')}.$$
Analogously let $y \in J_n$ be such that
$$|h(J_n)|=\frac{|R^ng(h(I^n_{\al}))|}{Dg^{i_n(x_0)}(h(y))}.$$
Therefore
\begin{eqnarray*}
\frac{|h(J_n)|}{|J_n|}&=&\frac{Df^{i_n(x_0)}(y')}{Dg^{i_n(x_0)}(h(y))}\cdot \frac{|R^ng(h(I^n_{\al}))|}{|R^nf(I^n_{\al})|}\\
&=&\frac{Df^{i_n(x_0)}(y')}{Df^{i_n(x_0)}(y)}\cdot e^{\psi_n(y)}\cdot \frac{|R^ng(h(I^n_{\al}))|}{|R^nf(I^n_{\al})|}.
\end{eqnarray*}
When  $n$ converges to  infinity we have, by  Lemma \ref{unifor}, Lemma \ref{const} and Lemma \ref{dist45}  that

\begin{equation} \label{form.lip} Dh(x_0)=C_{8}\cdot e^{\psi(x_0)}.\end{equation}
Writing
$$h(x)=\int_{0}^x Dh(s)\ ds=\int_{0}^x C_8 \cdot  e^{\psi(s)} \ ds,$$ we have the result.
\cqd

\begin{pro}  If  $f,g\in\mathcal{B}^{2+\nu}_{k,*}$ have the  same combinatorics  and they admit the same weak affine model then $h$ is differentiable at  every point $x_0$ such that $f^n(x_0)\neq 0$ for all $n\geq0$ and (\ref{form.lip}) holds. \end{pro}
\begin{proof} One can prove that $\psi$ is continuous at every point $x_0$ such that $f^n(x_0)\neq 0$ for all $n\geq0$. Indeed under this assumption we can carry out Cases {\it i} and {\it iv} in the proof of Lemma \ref{sol.coh} and Lemma \ref{afim.break}, as well  the proof of Proposition \ref{c_1}.
\end{proof}
\section{Linearization}
In this section we will show that, for each  $f \in \mathcal{B}^{2+\nu}_{k,*}$  there exists a {\it unique} weak affine  model that is $C^1$ conjugate with $f$.
\begin{lemma}\label{break.dif} Let $f_A$ and $f_B$ be two weak affine models  of $f$.  If $f_A$ and $f_B$  have the same breaks  then $f_A=f_B$. In particular if $f_A$ and $f_B$ are $C^1$ conjugate on the circle then $f_A=f_B$.
\end{lemma}
\begin{proof} If $f_A$ and $f_B$ have the  same breaks, the corresponding slope vectors $\omega^A$ and $\omega^B$ satisfy $\omega^A_i- \omega^B_i=v$, for every $i$.  Let $H$ be a conjugacy of $f_A$ with a rotation on the circle.  By  \cite[Proposition 2.3]{mmy}  we have
$$\sum_i \omega^A_i |H(I^A_i)| = \sum_i \omega^B_i |H(I^A_i)| =0,$$
which implies $\omega^A=\omega^B$. By Proposition \ref{affine} we have that $f_A=f_B$.
\end{proof}
\begin{proof}[Proof of Theorem \ref{teo5}] Let $\omega$ be as in Lemma \ref{log.deriv} and choose $v\in E^s_0\setminus\{0\}$.  By  Proposition \ref{affine} there is an unique weak affine model $g_t$ of $f$ with vector slope $\omega_t:= \omega+ t\cdot v$. By Lemma \ref{break.dif} the break at $0$ is a {\it non constant} linear functional on $t$. Let $t_0$ be the unique parameter such that the break at $0$  of $g_{t_0}$ coincides with the break at $0$ of $f$. Since by Lemma \ref{afim.break} they already have $d-2$ identical breaks and the product  of the breaks is $1$, it follows that all the breaks of $f$ and $g_{t_0}$ coincides. By Theorem \ref{teo4} the conjugacy between $f$ and $g_{t_0}$ is $C^1$. On the other hand, every piecewise affine homeomorphism $g$ of the circle that is $C^1$ conjugate with $f$ is a weak affine model of $f$ with the same break points of $f$, so $g=g_{t_0}$.
\end{proof}

\begin{rem} All weak affine models of the g.i.e.m. $f$ belongs the the one-parameter family $g_t$.  All of them are Lipchitz conjugate with $f$. Only one of these weak models, the {\it strong affine model} $g_{t_0}$, is indeed $C^1$ conjugate with $f$.
\end{rem}

\bibliographystyle{elsarticle-num}

\end{document}